\documentclass[11pt]{amsart}
\usepackage{amsfonts,latexsym,amsthm,amssymb,amsmath,amscd,euscript,tikz, tikz-cd}
\usepackage[alphabetic, msc-links, bibtex-style, nobysame]{amsrefs}
\usepackage{stackengine}
\usepackage{framed}
\usepackage{xfrac}
\usepackage[makeroom]{cancel}
\usepackage{faktor}
\usepackage{braket}
\usepackage{pgf,tikz,pgfplots}
\pgfplotsset{compat=1.17}
\usepgfplotslibrary{fillbetween} 
\usepackage{mathrsfs}
\usetikzlibrary{arrows}
\definecolor{wrwrwr}{rgb}{0.3803921568627451,0.3803921568627451,0.3803921568627451}
\definecolor{rvwvcq}{rgb}{0.08235294117647059,0.396078431372549,0.7529411764705882}
\definecolor{mblue}{rgb}{0.2, 0.3, 0.8}
\definecolor{morange}{rgb}{1, 0.5, 0}
\definecolor{mgreen}{rgb}{0.1, 0.4, 0.2}
\definecolor{mred}{rgb}{0.5, 0, 0}
\definecolor{ForestGreen}{RGB}{34,139,34}
\usepackage{hyperref}
\hypersetup{colorlinks=true,citecolor=ForestGreen,linkcolor = blue,urlcolor =black,linkbordercolor={1 0 0}}
\usepackage{float}

\numberwithin{equation}{section}

\usepackage{lmodern}
\usepackage{supertabular}
\usepackage{verbatim}
\usepackage{amssymb}
\usepackage{enumerate}
\usepackage{stmaryrd}
\usepackage{bbm}
\usepackage{mathtools}
\usepackage[nopatch=footnote]{microtype}
\usepackage{setspace}
\usepackage{tikz,tikz-cd}
\usetikzlibrary{matrix,calc,positioning,arrows,decorations.pathreplacing,patterns,knots}

\newcommand{\la}{\langle}
\newcommand{\rg}{\rangle}

\newtheorem{theorem}{{Theorem}}[section]
\newtheorem*{theorem*}{Theorem}
\newtheorem{lemma}[theorem]{Lemma}
\newtheorem{proposition}[theorem]{Proposition}

\newtheorem{corollary}[theorem]{Corollary}
\newtheorem*{corollary*}{Corollary}

\theoremstyle{definition}
\newtheorem{definition}{Definition}
\newtheorem{remark}{Remark}

\usepackage{lmodern,url,enumerate,mathtools,microtype}
\usepackage[hmargin = 0.95in,vmargin=0.95in]{geometry}
\usepackage{graphicx}
\usepackage{subcaption}

\newcommand{\ve}{\varepsilon}

\newcommand{\mr}[1]{{\rm #1}}
\newcommand{\cA}{\mathcal{A}}\newcommand{\cB}{\mathcal{B}}
\newcommand{\cC}{\mathcal{C}}
\newcommand{\cE}{\mathcal{E}}
\newcommand{\cG}{\mathcal{G}}

\newcommand{\cL}{\mathcal{L}}
\newcommand{\cM}{\mathcal{M}}
\newcommand{\cP}{\mathcal{P}}
\newcommand{\cR}{\mathcal{R}}
\newcommand{\cS}{\mathcal{S}}
\newcommand{\cU}{\mathcal{U}}

\newcommand{\bN}{\mathbb{N}}

\newcommand{\bR}{\mathbb{R}}
\newcommand{\bS}{\mathbb{S}}

\newcommand{\nc}{\newcommand}

\nc{\on}{\operatorname}
\nc{\p}{\partial}
\nc{\ol}{\overline}
\nc{\ul}{\underline}
\nc{\pa}{\partial}

\nc{\pb}{\partial_b}
\nc{\pc}{\partial_c}
\nc{\pd}{\partial_d}
\nc{\pe}{\partial_e}
\nc{\pf}{\partial_f}
\nc{\pg}{\partial_g}
\nc{\ph}{\partial_h}
\nc{\pari}{\partial_i}
\nc{\pj}{\partial_j}
\nc{\pk}{\partial_k}
\nc{\pl}{\partial_l}
\nc{\pell}{\partial_\ell}
\nc{\parm}{\partial_m}
\nc{\pn}{\partial_n}
\nc{\po}{\partial_o}
\nc{\pp}{\partial_p}
\nc{\pq}{\partial_q}
\nc{\pr}{\partial_r}
\nc{\ps}{\partial_s}
\nc{\pt}{\partial_t}
\nc{\pu}{\partial_u}
\nc{\pv}{\partial_v}
\nc{\pw}{\partial_w}
\nc{\px}{\partial_x}
\nc{\py}{\partial_x}
\nc{\pz}{\partial_z}

\nc{\Spec}{\on{Spec}}

\nc{\sn}{\mr{sn}}
\nc{\cn}{\mr{cn}}
\nc{\dn}{\mr{dn}}

\numberwithin{equation}{section}
\makeatletter
\@addtoreset{equation}{section}
\makeatother

\title{Mean curvature flows with prescribed singular sets}
\date{\today}

\author[Raphael Tsiamis]{Raphael Tsiamis}
\address{ \vspace{-0.1in} \newline 
Department of Mathematics, Columbia University
\newline {\href{mailto:r.tsiamis@columbia.edu}{r.tsiamis@columbia.edu}}}

\begin{document}

\begin{abstract}
For every closed set $K \subset \mathbb{R}^n$ and every $m \geq 2$, we construct a mean-convex ancient solution to mean curvature flow of hypersurfaces in $\mathbb{R}^{m+n}$, with respect to a smooth Riemannian metric arbitrarily $C^\infty$-close to the Euclidean metric, whose first-time singular set is exactly $K \times \{0\}$.
\end{abstract}

\vspace*{0.2in}
\maketitle


\section{Introduction}

Mean curvature flow (MCF) is the evolution of hypersurfaces by their mean curvature vector, providing a canonical method for deforming submanifolds that has been studied extensively since the foundational work of Brakke~\cite{brakke}.
A central theme is the formation and structure of singularities, especially for \textit{mean-convex} flows, i.e., flows with nonnegative mean curvature.
Deep structural results of White~\cites{white-size-singular, nature-singularities} and Huisken--Sinestrari~\cites{huisken, huisken-sinestrari-1} yield a qualitative picture of singularity formation for mean convex flows, and subsequent works develop a comprehensive theory of noncollapsing and surgery~\cites{andrews, brendle, brendle-huisken-1, metzger-schulze, huisken-sinestrari-2, haslhofer-kleiner-1, haslhofer-kleiner-2}.
This theory strongly constrains the possible geometry of singularities: White proved that the spacetime singular set of a mean-convex flow in $\bR^{n+1}$ has parabolic Hausdorff dimension at most $n-1$~\cite{nature-singularities}, and Colding-Minicozzi proved that this set is contained in a finite union of codimension-$2$ Lipschitz submanifolds together with a set of codimension at least $3$~\cites{colding-minicozzi, singular-set}.
Moreover, Cheeger--Haslhofer--Naber established a quantitative stratification and proved Minkowski-type estimates for the singular set of $k$-convex Brakke flows in arbitrary codimension~\cite{quantitative-stratification}.

Against this background, it is natural to ask how flexible the singular set at the first singular time can be.
A question attributed to Colding--Minicozzi, Ilmanen, and White asks whether any $C^1$ curve $\gamma\subset \bR^3$ can arise as the singular set of a mean-convex flow.
In~\cite{white-curves}*{\S 5}, White conjectured that the singular set of a mean curvature flow in $\bR^3$ may only consist of isolated points and curves.
At present, however, very few examples of first-time singular sets are known.
The basic model cases are $\bR^k\times\{0\}\subset\bR^{n+1}$ arising from the shrinking cylinder $\bR^k\times\bS^{n-k}$, and $\bS^{n-k}\times\{0\}$ arising from the ``marriage ring,'' an $O(n-k)\times O(k+1)$-invariant hypersurface diffeomorphic to $\bS^{n-k}\times\bS^k$ with a very small $\bS^k$-radius.
The only other known examples are an incomplete flow with first-time singular set the ray $(0,\infty)$~\cite{wang-self-shrinkers}, and non-product flows in $\bR^{m+n}$ with singular set $\bR^n \times \{ 0 \}$~\cite{huang-zhao}.

In this paper we show that, after an arbitrarily small smooth perturbation of the ambient Euclidean metric, the singular set at the first singular time can be prescribed with essentially no restriction beyond closedness.
More precisely, for any closed set $K \subset \bR^n$ we construct mean-convex ancient solutions of MCF in $\bR^{m+n}$ whose singular set at the first singular time is exactly $K \times \{0 \}\subset \bR^n \times \bR^m$.
\begin{theorem}\label{thm:main-theorem}
    Let $K$ be a non-empty closed subset of $\bR^n$, and $m \geq 2$.
    For each sufficiently small $\tau \in (0,\tau_0(m,n))$, there exists a function $f(x,r) \in C^{\infty}(\bR^n \times \bR_{\geq 0} )$ with 
    \[
    \sup |D_x^k \partial_r^{\ell} (f-1)| < C(m,n,k,\ell) \, \tau, \qquad \text{for every } \; k, \ell \geq 0,
    \]
    and a family of hypersurfaces $\{ \cG_t \}_{t \in (-\infty, 0)}$ that is a mean-convex ancient solution of mean curvature flow in $\bR^{m+n} = \bR^n_x \times \bR^m_{\xi}$ with respect to the metric 
    \[
    g_f(x,\xi) = \sum_{i=1}^n dx_i^2 + f (x,|\xi|^2) \sum_{j=1}^m d \xi_j^2, \qquad r := |\xi|^2,
    \]
    such that the singular set at the first singular time $t=0$ is exactly $\on{sing} \, \cG_t = K \times \{ 0 \} \subset \bR^n \times \bR^m.$
\end{theorem}
Thus, in every ambient dimension $m+n \ge 3$, the singular set at the first singular time can be arbitrarily complicated.
In particular, in $\bR^3$ we obtain mean-convex ancient solutions for which $K$ is fractal.
Since the ambient metric can be taken arbitrarily $C^\infty$-close to the Euclidean metric, this shows that the Euclidean structural picture for mean-convex singular sets established by~\cite{singular-set} is not stable under arbitrarily small smooth perturbations of the ambient metric.

The construction of Theorem~\ref{thm:main-theorem} utilizes an ansatz for $O(m)$-invariant graphical flows in a warped product background.
The main analytic ingredient is the construction of a solution to $O(m)$-invariant MCF on a non-cylindrical domain which stays on one side of a shrinking cylinder and converges to it exponentially; this is obtained in Theorem~\ref{thm:solution-of-smcf-with-estimates}.
We then solve a transport equation to produce a smooth Riemannian metric as an arbitrarily small perturbation of the Euclidean one, for which this solution can be glued smoothly with the shrinking cylinder.
The glued arrival time function describes a mean-convex ancient flow whose singular set at the first singular time is exactly $K\times\{0\}$.

Our result parallels Leon Simon's recent breakthrough construction of stable minimal hypersurfaces in $\bR^{N+1+\ell}$ with ambient metric arbitrarily close to Euclidean and with singular set an arbitrary closed set $K \times \{ 0 \}$~\cite{prescribed-singular}.
However, the approximately cylindrical solution of Theorem~\ref{thm:solution-of-smcf-with-estimates} is obtained using different methods.
We first construct a very precise approximate solution of $O(m)$-invariant mean curvature flow together together with barriers that remain close on the entire non-cylindrical region, which we approximate by an infinite staircase of cylindrical domains.
We then prove the existence of a solution by an inductive argument, utilizing crucially the parabolic version of Savin's small perturbation theorem~\cites{savin-small-perturbation , small-perturbation } to control the solution in a smaller domain.
As discussed in Section~\ref{section:non-cylindrical}, this method is flexible and applies to a broader class of quasilinear and non-linear parabolic problems, notably the construction of prescribed singular sets for other geometric flows.

\subsection{Acknowledgments}
I am grateful to my advisor, Simon Brendle, for suggesting this problem and for his invaluable guidance and continued support. 
I am also thankful to Ovidiu Savin for helpful discussions.
This work was supported in part by the A.G.~Leventis Foundation Scholarship and the Onassis Foundation Scholarship.

\section{Mean curvature flows in a non-cylindrical domain}\label{section:non-cylindrical}

For a function $u(x,t)$ defined over a spacetime domain $\Omega \subset \bR^n_x \times (-\infty, 0]$, we consider the family of hypersurfaces $\cG_t \subset \bR^{m+n}$ given by $O(m)-$invariant tubes of radius $u(x,t)$, i.e., the symmetric graphs
\[
    \cG_t := \text{\sffamily{SG}}( u(-,t)) = \{ (x,t,\xi) \in \Omega \times \bR^m : |\xi| = u(x,t) \}.
\]
The hypersurfaces $\cG_t$ have normal vector $\nu = \frac{( -\nabla u, \xi / |\xi|)}{\sqrt{1 + |\nabla u|^2}}$, so their mean curvature is given by
\[
H = \on{div} \nu = \frac{1}{\sqrt{1 + |\nabla u|^2}} \left( - \Delta u + \frac{Q(u)}{1 + |\nabla u|^2} + \frac{m-1}{u} \right).
\]
The mean curvature flow equation $(\partial_t F)^{\perp} = - H \nu$ for the cylindrical graphs $\cG_t$ therefore amounts to
\begin{equation}\label{eqn:cmcf}\tag{CMCF}
    \hat{\cM}(u) := \partial_t u - \Delta u + \frac{Q(u)}{1 + |\nabla u|^2} + \frac{m-1}{u} = 0,
\end{equation}
where $\nabla u$ and $\Delta u$ respectively denote the Euclidean gradient and Laplacian of the function.
We also define $Q(u) := D^2 u (\nabla u, \nabla u) = \sum_{i,j} u_i u_j u_{ij}$.
Hereafter, we refer to equation~\eqref{eqn:cmcf} as cylindrical mean curvature flow (CMCF); taking $u(x,t) = \hat{\varphi}_{\lambda}(t) := \sqrt{\lambda -2(m-1)t}$ recovers the shrinking cylinder solutions $\cC_{\lambda} :=\bR^n \times \bS^{m-1}_{\hat{\varphi}_{\lambda}(t)}$, which become singular at time $t = \frac{\lambda}{2(m-1)}$.

We are interested in solutions of equation~\eqref{eqn:cmcf} that remain $C^k$-close to the cylinder solution $\hat{\varphi}_0(t) = \sqrt{-2(m-1) t}$, so it is useful to formulate the problem in terms of the function $w = u^2$.
We observe that $2u \, \hat{\cM}(u) = \cM_1(w)$, where $\cM_{\zeta}(w)$ denotes the operator
\begin{equation}\label{eqn:cmcf-for-u2}\tag{$\textup{SMCF}_{\zeta}$}
    \cM_{\zeta}(w) := \partial_t w - \Delta w + 2(m-1) + \zeta(x,t)\frac{Q(w) + 2 \, |\nabla w|^2}{4 w + |\nabla w|^2}
\end{equation}
for a given spacetime function $\zeta(x,t)$.
Notably, $\cM_0(w)$ is a heat operator with constant forcing and the cylinder solutions correspond to $\varphi_{\lambda}(t) = \lambda - 2(m-1) t$, with $\cM_{\zeta} ( \varphi_{\lambda}) = 0$ for arbitrary $\zeta$.
As in graphical mean curvature flow, the operator $\cM_{\zeta}$ satisfies a comparison principle provided that $\zeta \leq 1$.
\begin{lemma}\label{lemma:comparison-principle} 
Consider two functions $w_1, w_2 \in C^2( \bar{W})$ defined on a spacetime domain
\[
    W := \{ (x,t) : x \in U, \;  q(x) < t < T \} \, .
\]
The parabolic boundary of $W$, denoted $\partial_p W$, is given by
\[
\partial_p W := \{ (x, q(x)) : x \in U \} \cup \{ (x,t) :x \in \partial U, \; q(x) \leq t < T \}.
\] 
Suppose that $\cM_{\zeta}(w_1) \leq \cM_{\zeta}(w_2)$ in $W$ and $w_1 \leq w_2$ on $\partial_p W$, with $\zeta \leq 1$.
    Then, $w_1 \leq w_2$ in $W$ and $w_1 - w_2$ cannot attain a zero maximum outside of the parabolic boundary $\partial_p W$ unless $w_1 = w_2$.
    Moreover, $w_1 - w_2 \leq \sup_{\partial_p W} (w_1 - w_2)$.
\end{lemma}
\begin{proof}
The operator $\hat{\cM_{\zeta}}$ of~\eqref{eqn:cmcf-for-u2} is expressible as $\cM_{\zeta}= \partial_t w - a^{ij}(w, \nabla w) w_{ij} + b(w, \nabla w)$, where $b(w, \nabla w) = 2(m-1) + 2 \zeta \frac{|\nabla w|^2}{4w + |\nabla w|^2}$ and $a^{ij}(w, \nabla w) = \delta_{ij} - \zeta \frac{w_i w_j}{4 w + |\nabla w|^2}$ has minimum eigenvalue $\geq 1 - \zeta \frac{|\nabla w|^2}{4 w + |\nabla w|^2} = \frac{4w + (1-\zeta) |\nabla w|^2}{4w + |\nabla w|^2}$, so it is positive semidefinite for $\zeta \leq 1$.
Consequently, the operator $\cM_{\zeta}$ is quasilinear parabolic on $\{ w > 0 \}$, so it satisfies the standard comparison principle.
This property holds for non-compact, non-cylindrical domains, including $W$, by work of Athanassenas-Kandanaarachchi \cite{a-k}*{Proposition 3.1 \& Appendix A}.
\end{proof}

In all the steps that follow, we fix an arbitrary non-empty closed subset $K \subset \bR^n$ and denote $\tau \in (0, \frac{1}{4} ]$ sufficiently small, to be chosen later, depending only on $n,m$.
\begin{definition}\label{def:k-cutoff-h}
    Let $K \subset \bR^n$ be an arbitrary non-empty closed subset. 
    There exist constants $C(j,k)$ such that for any $\tau \in \bigl(0, \frac{1}{4} \bigr]$, we have a cutoff function $h$ such that 
    \begin{equation}\label{eqn:h-properties}
    \begin{cases}
        h \equiv 0 &\text{on } K, \\
        h > 0 &\text{on }U = \bR^n \setminus K, \\
        \dfrac{|D^k h(x)|}{\on{dist}(x, \partial U)^j} \leq C(j,k) \, \tau & \text{for all } x \in \bR^n \; \text{and all } \; j,k \in \bN_0, \\
        |h|^{\frac{1}{4}} + |h|^{-\frac{3}{4}} |\nabla h| + |h|^{- \frac{1}{2}} |D^2 h| + |h|^{- \frac{1}{4}} |D^3 h| < \tau & \text{for all } x \in \bR^n.
    \end{cases}
    \end{equation}
\end{definition}
\noindent The existence of a function $\tilde{h}$ satisfying the first three properties with $\sum_{i=0}^3 |D^i \tilde{h}| < \frac{\tau}{10^3}$ is standard, cf.~\cites{lieberman-distance , stein-regularized-distance}.
Then, the function $h := \tilde{h}^4$ also satisfies the fourth property.
Working with such a function, the property~\eqref{eqn:h-properties} ensures that $h^{\frac{1}{4}}$ is $\tau$-Lipschitz.
We define the spacetime set
\begin{equation}\label{eqn:spacetime-set}
   W = \left\{ (x,t) \in \bR^n \times (-\infty, 0] : x \in U, \; - \tfrac{1}{100 m} h(x)^4 < t < 0 \right\}, \qquad U := \bR^n \setminus K. 
\end{equation}
The main result of this section, which forms a crucial piece of Theorem~\ref{thm:main-theorem}, is the existence of a solution to equation~\eqref{eqn:cmcf} in $W$ decaying exponentially to the cylinder.
\begin{theorem}\label{thm:solution-of-smcf-with-estimates}
    Let $\delta \in (0, \delta_0)$.
    There exists a $\tau_0(\delta, m,n) \in (0, \frac{1}{2} ]$ such that for $\tau \in (0,\tau_0]$, $h$ as in Definition~\ref{def:k-cutoff-h}, and $W$ as in \eqref{eqn:spacetime-set}, there exists a function $w_{\tau} \in C^{\infty}( \bar{W} \setminus (\partial U \times \{0\}) ) \cup C^0(\bar{W})$ that solves $\cM_1 (w) = 0$ and satisfies, for all $(x,t) \in W$,
    \begin{equation}\label{eqn:wt-conditions}
        \begin{split}
            &w_{\tau}^{-1}(x,t) \, |\nabla_x w_{\tau}(x,t)|^2 + |D^2_x w_{\tau}(x,t)| < \delta \qquad \text{and} \qquad \partial_t w_{\tau}(x,t) < -1, \\
            &\exp \bigl( - h(x)^{-2} \bigr)< w_{\tau}(x,t) + 2(m-1) t < \exp \bigl( - h(x)^{- \frac{11}{10}} \bigr).
        \end{split}
    \end{equation}
\end{theorem}
The proof of this Theorem is given in Section~\ref{section:construction-of-a-solution}, after constructing an approximate solution and barrier functions for $\cM_1$ in Section~\ref{section:barriers}.
These results produce a fairly precise description of the solution $w_{\tau}$: it remains $C^k$-close to the cylinder $\varphi_1(s) = 1 - 2(m-1) s$ at scale $\rho_0 = \sqrt{w_{\tau}(x_0,t_0)}$.
First, we use Definition~\ref{def:k-cutoff-h} to collect some geometric properties of $h$ and the domain $W$.
\begin{lemma}\label{lemma:geometric-properties-of-h}
    For any $M > 0$, there exists a small $\tau_0(m,n,M)$ such that for any $\tau \in (0,\tau_0]$, we have
    \begin{align}
    &\rho_0 \leq M h(x_0)^{\frac{4}{3}} \implies \rho_0 \leq \bigl( 1 + \tau h(x_0) \bigr) \, M \inf_{x \in B_{\rho_0}(x_0)} h(x)^2, \label{eqn:r0-inequality} \\
    &(1 - \tau h(x_0) \bigr) \, h(x_0)^2 \leq h(x)^2 \leq (1 - \tau h(x_0) \bigr) \, h(x_0)^2 \qquad \text{in } \; B_{\rho_0}(x_0). \label{eqn:small-enclosure-h(x0)}
    \end{align}
    Moreover, for $c < \frac{1}{10 m}$ and any spacetime point $(x_0,t_0)$ with $t_0 > - c h(x_0)^4$, we have
    \begin{equation}\label{eqn:cylinder-containment-in-enlarges}
    P_{\rho_0}(x_0,t_0) := B_{\rho_0}(x_0) \times [t_0 - r^2, t_0] \subset \{ (x,t) : x \in U, - 4 mc \, h(x)^4 < t <  0 \}
    \end{equation}
    whenever $\rho_0 \leq \sqrt{2mc} \, h(x_0)^2$.
\end{lemma}
\begin{proof}
The function $h^{\frac{1}{4}}$ is $\tau$-Lipschitz by Definition~\ref{def:k-cutoff-h}, so for any $x \in B_{\rho_0}(x_0)$ we have
\begin{align*}
    &|h(x)^{\frac{1}{4}} - h(x_0)^{\frac{1}{4}}| \leq \tau \rho_0 \leq \tau M h(x_0)^{\frac{4}{3}} \implies \bigl( 1 - \tau  h(x_0)^{\frac{13}{12}} \bigr)^4 h(x_0) \leq h(x) \leq \bigl( 1 + \tau  h(x_0)^{\frac{13}{12}} \bigr)^4 h(x_0)
\end{align*}
from which the bound~\eqref{eqn:small-enclosure-h(x0)} follows, also proving~\eqref{eqn:r0-inequality}.
For the property~\eqref{eqn:cylinder-containment-in-enlarges}, we use~\eqref{eqn:small-enclosure-h(x0)} to find $h(x)^4 > (1 - 4 \tau h(x_0)) h(x_0)^4$ for all $x \in B_{\rho_0}(x_0)$.
We therefore obtain, for all $(x,t) \in P_{\rho_0}(x_0,t_0)$,
\[
|t| \leq |t_0 - \rho_0^2| < c h(x_0)^4 + \rho_0^2 < c h(x_0)^4 + 2m c h(x_0)^4 < 3mc (1 - 4\tau h(x_0) ) h(x_0)^4,
\]
for $|h|_{L^{\infty}(U)} \leq \tau_0$ small.
Thus, $|t| < 4mc \, h(x)^4$ proves the claimed containment~\eqref{eqn:cylinder-containment-in-enlarges}.
\end{proof}

\subsection{Construction of barriers}\label{section:barriers}

We first construct upper and lower barriers $w_{\pm}$ for the equation $\cM_{\zeta}(w) = 0$, with arbitrary $0 \leq \zeta \leq 1$, that remain close to the cylinder $\varphi_1(s)$ in a rescaled sense.
Next, we will produce a solution $w_{\tau}$ of the modified equation $\cM_{\zeta} (w) = 0$ that remains close to $\varphi_1(s)$ at scale $\rho_0 = \sqrt{w_{\tau}(x_0,t_0)}$, allowing us to invoke the parabolic counterpart of Savin's \textit{small perturbation theorem}~\cites{savin-small-perturbation , small-perturbation} to prove that $w_{\tau}$ is smooth and satisfies the estimates~\eqref{eqn:wt-conditions}.

\begin{lemma}\label{lemma:subsolution}
    Let $F(x,t)$ be a positive function defined on an open subset of $\bR^n_x \times (-\infty,0]_t$ and satisfying $|D^2 F| \leq 1$.
    Then, for any function $\zeta$ with $0 \leq \zeta \leq 1$, we have that
    \begin{align}
        \tfrac{1}{n} \partial_t F > |D^2 F| + F^{-1} |\nabla F|^2 \quad & \implies \quad \cM_{\zeta}( - 2(m-1) t + F) > 0, \label{eqn:barrier-condition-supersolution} \\
        - \tfrac{1}{n} \partial_t F < |D^2 F| + F^{-1} |\nabla F|^2 \quad & \implies \quad \cM_{\zeta}( - 2(m-1) t + F) < 0 \, . \label{eqn:barrier-condition-subsolution}
    \end{align}
\end{lemma}
\begin{proof}
    For any function $w$, we denote by $|D^2 w|$ the operator norm of the Hessian and bound 
    \begin{equation}\label{eqn:control-nonlinear-terms}
\begin{split}
    - |\nabla w|^2 |D^2 w| &\leq Q(w) \leq |\nabla w|^2 |D^2 w|, \qquad \text{and} \qquad - n |D^2 w| \leq \Delta w \leq n |D^2 w|, \\
    &\left| \cM_{\zeta} w - \partial_t w - 2(m-1) \right| \leq |\Delta w| + \frac{|Q(w)| + 2 |\nabla w|^2}{4 w}.
\end{split}
\end{equation}
Taking the function $w = - 2(m-1) t + F$, we compute $\partial_i w = \partial_i F$ and $\partial_t w + 2(m-1) = \partial_t F$.
Combining this inequality with the bounds $w \geq F$ and $|D^2 F| \leq 1$, we find
\[
|\Delta w| + \tfrac{1}{4w} ( Q(w) + 2 |\nabla w|^2 ) \leq n \, |D^2 F| + F^{-1} \bigl( |D^2 F| + 1) |\nabla F|^2  \leq n |D^2 F| + |F|^{-1} |\nabla F|^2.
\]
Using the estimate~\eqref{eqn:control-nonlinear-terms}, we conclude that $\frac{1}{n} |\cM_{\zeta} w - F| \leq |D^2 F| + |F|^{-1} |\nabla F|^2$, hence
\[
\on{sgn}(\partial_t F) \cM_{\zeta}w > \on{sgn}(\partial_t F) \partial_t F - n ( |D^2F| + F^{-1} |\nabla F|^2).
\]
We can now apply the bound~\eqref{eqn:barrier-condition-supersolution} or~\eqref{eqn:barrier-condition-subsolution} for $F$ to obtain $\on{sgn}(\partial_t F) \cM_{\zeta}w > 0$ as claimed.
\end{proof}

\begin{corollary}\label{cor:general-barrier-functions}
For $\tau \in (0,\tau_0]$ sufficiently small, we have functions $w_{\pm}$ with
\begin{align*}
    w_{\pm} &:= -2(m-1)t + \,F_{\pm}, \qquad F_{\pm} (x,t) := \bigl[ h(x)^3 \pm ( h(x)^4 + t) \bigr] \exp( - h(x)^{- \frac{9}{8}}), \\
    &\cM_{\zeta}( w_- ) < 0 < \cM_{\zeta}( w_+) \qquad \text{in } \; \{ (x,t) : - h(x)^4 < t < 0 \}, \qquad \text{for any } \; 0 \leq \zeta \leq 1.
\end{align*}
\end{corollary}
\begin{proof}
    We will prove the claimed properties of $w_{\pm}$ by verifying the conditions~\eqref{eqn:barrier-condition-supersolution} and~\eqref{eqn:barrier-condition-subsolution} for $F_{\pm}$.
    Writing $\phi(s) := \exp( - s^{- \frac{9}{8}})$, with $\phi,\phi',\phi'' > 0$, we compute that $\partial_t F_{\pm} = \pm \phi$.
    Using the chain rule together with $0 \leq h(x)^4 + t \leq h(x)^4$ and $|h| < \tau$ in $W$, we can bound $F_{\pm} \geq \tfrac{1}{2} h(x)^3 \phi(h(x))$ and
\begin{align*}
    |F_{\pm}|^{-1} |\nabla F_{\pm}|^2 &\leq 200 h(x) \phi^{-1} ( \phi^2 + h(x)^2 (\phi')^2 ) |\nabla h|^2, \\
    |D^2 F_{\pm}| &\leq 40 h(x) \left[ ( \phi + h(x) \phi' + h(x)^2 \phi'') |\nabla h|^2 + h(x) ( \phi + h(x) \phi') |D^2 h| \right].
\end{align*}
Consequently, the properties~\eqref{eqn:barrier-condition-supersolution} and~\eqref{eqn:barrier-condition-subsolution} will be proved upon establishing that
    \[
\tfrac{1}{400 n} \phi \geq h(x) \left[  ( \phi + h(x) \phi' ) h |D^2 h| + \left( \phi + h(x) \phi'+ h(x)^2 \phi^{-1}(\phi')^2 + h(x)^2  \phi'' \right) |\nabla h|^2 \right].
\]
Recall from ~\eqref{eqn:h-properties} that $\sum_{k=0}^3 |h|^{\frac{1}{4} -k} |D^k h| < \tau$.
    Denoting $s := h(x)$ and taking $\tau$ sufficiently small in Definition~\ref{def:k-cutoff-h}, we therefore see that the desired properties~\eqref{eqn:barrier-condition-supersolution} and~\eqref{eqn:barrier-condition-subsolution} will follow from
    \[
    C \phi \geq s^{\frac{5}{2}} \left[ \phi + s \phi' + s^2 ( \phi^{-1} (\phi')^2 +  \phi'' )\right]
    \]
    for $C$ a universal constant and $s \in (0,\tau)$ sufficiently small; this holds for $\phi(s) = \exp( - s^{- \frac{9}{8}})$.
    Since the inequalities~\eqref{eqn:barrier-condition-supersolution} and~\eqref{eqn:barrier-condition-subsolution} are invariant under multiplication by small constants, the claim follows from Lemma~\ref{lemma:subsolution} after taking $\tau$ sufficiently small to make $|D^2 F_{\pm}| < 1$.
\end{proof}

\begin{lemma}\label{lemma:tight-barrier-estimates}
For any $a \in (0,1)$, there is a small $\tau_0(m,n,a)$ and some $c,C(m,n,a,U)$ such that on any parabolic cylinder $P_{\rho_0}(x_0,t_0) := B_{\rho_0}(x_0) \times [t_0 - \rho_0^2, t_0] \subset W$ as in~\eqref{eqn:spacetime-set}, we have the inequalities
\begin{equation}\label{eqn:u-plus-u-minus-refined-inequality}
    \begin{split}
        w_-(x,t) &\geq \bigl( 1 - C h(x_0)^a \bigr) \, \left[ w_+(x_0,t_0) - 2 (m-1) (t - t_0) \right] \, , \\
        w_+(x,t) &\leq \bigl( 1 + C h(x_0)^a \bigr) \, \left[ w_-(x_0,t_0) - 2(m-1) (t - t_0) \right] \, ,
    \end{split}
    \end{equation}
    whenever $\rho_0 \leq c \min \{ h(x_0)^{a + \frac{11}{8}} , h(x_0)^{\frac{a+3}{2}} \}$.
    In particular, this holds for $a = \frac{1}{10}$ and $\rho_0 < h(x_0)^{\frac{9}{5}}$.
\end{lemma}
\begin{proof}
Since $\min \{ \frac{11}{8}, \frac{3}{2} \} > \frac{4}{3}$, we have $\rho_0 \leq M h(x_0)^{\frac{4}{3}}$ so the bound~\eqref{eqn:small-enclosure-h(x0)} of Lemma~\ref{lemma:geometric-properties-of-h} applies.
Moreover, $\sum_{k=0}^3 |h|^{\frac{1}{4} - k} |D^k h| < \tau$.
Using the mean value theorem, we find
\begin{equation}\label{eqn:mvt-bound}
|h(x) - h(x_0)| \leq \sup_{B_{\rho_0}(x_0)} |\nabla h| \cdot |x - x_0| \leq C \tau \sup_{B_{\rho_0}(x_0)}|h|^{\frac{3}{4}} \rho_0 \leq C \tau h(x_0)^{\frac{3}{4}} \rho_0
\end{equation}
after bounding $h(x) \in [ \frac{1}{2} h(x_0), \frac{3}{2} h(x_0)]$ by~\eqref{eqn:small-enclosure-h(x0)}.
Differentiating the functions $\exp( - s^{- \frac{9}{8}})$ and $s^3(1-s) \exp( - s^{- \frac{9}{8}})$, we obtain the following bounds for $s \in (0,\tau_0]$ sufficiently small,
\begin{equation}\label{eqn:exp-bounds}
    \Bigl| \tfrac{d}{ds} \exp \bigl( - s^{- \frac{9}{8}} \bigr) \Bigr| < 2 s^{- \frac{17}{8}} \exp \bigl( - s^{- \frac{9}{8}} \bigr), \qquad \Bigl| \tfrac{d}{ds} \bigl[ s^3 (1-s) \exp \bigl( - s^{- \frac{9}{8}} \bigr) \bigr]  \Bigr| <  2 s^{\frac{7}{8}} \exp \bigl( - s^{- \frac{9}{8}} \bigr) \, .
\end{equation}
Writing $w_{\pm} = - 2(m-1)t + F_{\pm}$, we seek to prove that
    \begin{align*}
        - 2(m-1) t + F_- (x,t) &\geq (1 - C \, h(x_0)^a ) \bigl( - 2(m-1) t + F_+(x_0,t_0) \bigr), \\
        - 2(m-1) t + F_+(x,t) &\leq (1 + C \, h(x_0)^a ) \bigl( - 2(m-1) t + F_-(x_0,t_0) \bigr) \, .
    \end{align*}
We eliminate the term $- 2(m-1) C h(x_0)^a t \geq 0$ from each side and use $|h|_{L^{\infty}(U)} \leq \tau_0$ to bound
\[
F_+(x,t) - F_-(x,t) \leq 2 h(x)^4 \exp ( - h(x)^{- \frac{9}{8}}), \qquad F_-(x,t) > \tfrac{1}{2} h(x)^3 \exp ( - h(x)^{- \frac{9}{8}}),
\]
using $2 (1 - C h(x_0)^a) h(x_0)^4 < \frac{1}{4} C h(x_0)^{a+3}$ for $a \in (0,1)$.
By symmetry of the $F_{\pm}$, it suffices to prove
    \begin{equation}\label{eqn:suffice-to-prove-minus}
        F_-(x,t) - F_-(x_0,t_0) + \tfrac{1}{4} C h(x_0)^{a + 3} \exp \bigl( - h(x_0)^{- \frac{9}{8}} \bigr) > 0 \, .
    \end{equation}
Using the estimate~\eqref{eqn:mvt-bound}, we observe that
    \begin{align*}
    \bigl| (h(x)^3 - h(x)^4 ) - (h(x_0)^3 - h(x_0)^4) \bigr| &\leq 3 \sup_{P_{\rho_0}(x_0)} |h|^2 \cdot |h(x) - h(x_0)| \leq C h(x_0)^{\frac{11}{4}} \rho_0, \\
    \bigl| h(x)^{- \frac{9}{8}} - h(x_0)^{- \frac{9}{8}} \bigr| &\leq \tfrac{9}{8} \sup_{P_{\rho_0}(x_0)} |h|^{- \frac{17}{8}} |h(x) - h(x_0)| \leq C h(x_0)^{- \frac{11}{8}} \rho_0, 
    \end{align*}
    which is bounded by a constant for $\rho_0 \leq c h(x_0)^{\frac{11}{8}}$.
    Bounding $\frac{e^s-1}{s} \leq C$ for $s \in ( 0 , \tau)$, we find
    \allowdisplaybreaks{
    \begin{align*}
        &\left| \exp \bigl( - h(x)^{- \frac{9}{8}} \bigr) - \exp \bigl( - h(x_0)^{- \frac{9}{8}} \bigr) \right| \leq \exp \bigl( - h(x_0)^{- \frac{9}{8}} \bigr) \left( \exp \bigl( \bigl| h(x)^{- \frac{9}{8}} - h(x_0)^{- \frac{9}{8}} \bigr| \bigr) - 1 \right) \\
        & \qquad \qquad\qquad \qquad\qquad\qquad\qquad \quad \; \leq C \, \exp \bigl( - h(x_0)^{- \frac{9}{8}} \bigr) \cdot h(x_0)^{- \frac{11}{8}} \rho_0 \, , \\
        & \left| \bigl( h(x)^3 - h(x)^4  \bigr) \exp \bigl( - h(x)^{- \frac{9}{8}} \bigr) - \bigl( h(x_0)^3 - h(x_0)^4 \bigr) \exp \bigl( - h(x_0)^{- \frac{9}{8}} \bigr) \right| \\
        &\leq \left| \bigl( h(x)^3 - h(x)^4  \bigr) - \bigl( h(x_0)^3 - h(x_0)^4 \bigr) \right| \exp \bigl( - h(x_0)^{- \frac{9}{8}} \bigr) \\
        & \quad + \bigl( h(x)^3 - h(x)^4 \bigr) \Bigl| \exp \bigl( - h(x)^{- \frac{9}{8}} \bigr) - \exp \bigl( - h(x_0)^{- \frac{9}{8}} \bigr) \Bigr| \\
        &\leq \left( C h(x_0)^{\frac{11}{4}} \rho_0 + C h(x_0)^3 \cdot h(x_0)^{- \frac{11}{8}} \rho_0 \right) \exp \bigl( - h(x_0)^{- \frac{9}{8}} \bigr) \leq C h(x_0)^{\frac{13}{8}} \rho_0 \exp \bigl( - h(x_0)^{- \frac{9}{8}} \bigr) \, .
    \end{align*}}
    Here, we bounded $h(x) \in [ \frac{1}{2} h(x_0), \frac{3}{2} h(x_0)]$ by~\eqref{eqn:small-enclosure-h(x0)}.    
    Combining these estimates, we arrive at
    \begin{align*}
        F_-(x,t) - F_-(x_0,t_0) &= \bigl[ h(x)^3 - (h(x)^4 + t) \bigr] \exp \bigl( - h(x)^{- \frac{9}{8}} \bigr) - \bigl[ h(x_0)^3 - ( h(x_0)^4 + t_0) \bigr] \bigl( - h(x_0)^{- \frac{9}{8}} \bigr) \\
        &\geq - \left| \bigl( h(x)^3 - h(x)^4  \bigr) \exp \bigl( - h(x)^{- \frac{9}{8}} \bigr) - \bigl( h(x_0)^3 - h(x_0)^4 \bigr) \exp \bigl( - h(x_0)^{- \frac{9}{8}} \bigr) \right| \\
        & \quad - |t - t_0| \exp \bigl( - h(x_0)^{- \frac{9}{8}} \bigr) - |t_0| \cdot \left| \exp \bigl( - h(x)^{- \frac{9}{8}} \bigr) - \exp \bigl( - h(x_0)^{- \frac{9}{8}} \bigr) \right| \\
        &\geq - (C h(x_0)^{\frac{13}{8}} + \rho_0 ) \rho_0  \exp \bigl( - h(x_0)^{- \frac{9}{8}} \bigr) \, ,
    \end{align*}
    where in the last step we used $|t| < h(x)^4$ and $h(x_0)^4 h(x_0)^{- \frac{11}{8}} = h(x_0)^{\frac{21}{8}} \ll h(x_0)^{\frac{13}{8}}$.
    Thus, inequality~\eqref{eqn:suffice-to-prove-minus} will follow from $\tfrac{1}{4} C h(x_0)^{a+3} - \bigl( C h(x_0)^{\frac{13}{8}} + \rho_0 \bigr) \rho_0 > 0$.
    This property holds whenever $\rho_0 \leq c \min \{ h(x_0)^{a + \frac{11}{8}} , h(x_0)^{\frac{a+3}{2}} \}$, and in particular for $a = \frac{1}{10}$ and $\rho_0 < h(x_0)^{\frac{9}{5}}$ as desired.
\end{proof}
We now derive higher derivative estimates for solutions of equation~\eqref{eqn:cmcf-for-u2} using the parabolic small perturbation theorem.
In what follows, we study an operator $\cM_{\zeta}$ as in~\eqref{eqn:cmcf-for-u2} with $\zeta(x,t) := \eta \bigl( \frac{t}{h(x)^4} \bigr)$, for $\eta : \bR \to [0,1]$ a smooth cutoff function such that $\eta \equiv 0$ for $|\sigma|>1$ and $\Bigl| \frac{d^k}{d \sigma^k} \eta(\sigma) \Bigr| \leq C_k$.
For a solution $w$ of $\cM_{\zeta} w = 0$ in the parabolic cylinder $P_{\rho_0}(x_0,t_0) :=B_{\rho_0}(x_0) \times [t_0 - \rho_0^2, t_0] \subset \{ - h(x)^4 < t < 0 \}$, we consider $(z,s) \in B_1(0) \times [-1,0]$ and let $\tilde{\zeta}(z,s) := \zeta(x_0 + \rho_0 z, t_0 + \rho_0^2 s)$ in $P_1$.
Let $u(z,s) := \rho_0^{-2} w(x_0 + \rho_0 z , t_0 + \rho_0^2 s)$ be the function in the rescaled variables, so that
\begin{equation}\label{eqn:w-u-rescaling}
\partial_t w = \partial_s u, \qquad \Delta_x w = \Delta_z u, \qquad |\nabla_x w|^2 = \rho_0^2 |\nabla_z u|^2 , \qquad Q(w) = \rho_0^2 Q_z(u).
\end{equation}
Therefore, $u$ satisfies $F_{\tilde{\zeta}}( D^2 u, \nabla u, u,z,s) - u_s = 0$ in $P_1$ in the notation of~\cite{small-perturbation}, where
\begin{equation}\label{eqn:F-tilde-operator}
F_{\tilde{\zeta}} (M,p,\xi,z,s) := \on{tr} M - 2(m-1) - \tilde{\zeta}(z,s) \frac{\la Mp, p \rg + 2 |p|^2}{4\xi + |p|^2} \, , \quad \xi = \xi(z,s).
\end{equation}
Being able to invoke the small perturbation theorem~\cite{small-perturbation}*{Theorem 1.1 and Corollary 1.2} amounts to the validity of the five conditions $\mathbf{H}_{\varphi}1 ) - \mathbf{H}_{\varphi} 5)$ presented therein with regards to the pair $( \tilde{F}_{\rho_0}, \varphi_1)$ in the parabolic cylinder $P_1$, for $\varphi_1(s) = 1 - 2(m-1) s$ the cylinder and $\cU_{\delta}(\varphi_1)$ the admissible class
\[
\cU_{\delta}(\varphi_1) := \bigl\{ (M + D^2 \varphi_1, p + D \varphi_1, \xi + \varphi_1, z, s) : \| M\|, |p|, |\xi| < \delta, \; (z,s) \in P_1 \bigr\} .
\]
\begin{proposition}\label{prop:apply-small-perturbation}
    Fix $\delta \in (0,\frac{1}{2})$ and $\alpha \in (0,1)$.
    Let $P_r(x_0,t_0) := B_r(x_0) \times [t_0 - r^2, t_0]$ denote a parabolic cylinder centered at $(x_0,t_0)$ and contained in $\{ (x,t) : - h(x)^4 < t < 0 \}$.
    There exists a constant $\mu_*$ depending only on $(n,m,\alpha,\delta)$ such that the following holds.
    Let $\tau \in (0,\tau_0(m,n,\alpha,\delta)]$ be sufficiently small in Definition~\ref{def:k-cutoff-h}
    and suppose that $w$ satisfies
    \[
    \cM_{\zeta}(w) = 0 \quad \text{in } \; P_{\rho_0}(x_0,t_0), \qquad \sup_{(x,t) \in P_{\rho_0}(x_0,t_0)} |w(x,t) - \varphi_{\rho_0^2}(t-t_0)| \leq \mu_* \rho_0^2, 
    \]
    where $\rho_0 \leq C(m,n) h(x_0)^2$ and $\varphi_{\rho_0^2}(t) = \rho_0^2 - 2(m-1)t$.
    Then, $w \in C^{2,\alpha}(P_{\rho_0/2}(x_0,t_0))$ and
    \begin{align*}
       \sup_{P_{\rho_0/2}(x_0,t_0)} \left( \rho_0^{-2} | w - \varphi_{\rho_0^2}(t-t_0) | + \rho_0^{-1} |\nabla_x w| + |D^2_x w| + |\partial_t w + 2(m-1)| \right) &\leq \delta. 
    \end{align*}
\end{proposition}
\begin{proof}
We work in the rescaled framework considered above, letting $u(z,s) := \rho_0^{-2} w(x_0 + \rho_0 z , t_0 + \rho_0^2 s)$ so that $u$ is defined in the parabolic cylinder $P_1$ with the scaling properties~\eqref{eqn:w-u-rescaling} and satisfies equation~\eqref{eqn:F-tilde-operator}.
For $\varphi_1(s)= 1 - 2(m-1) s$ the cylinder solution, our assumptions amount to
\[
    F_{\tilde{\zeta}}(D^2u, \nabla u,u,z,s) - u_s = 0 \quad \text{in } \; P_1, \qquad \| u - \varphi_1 \|_{L^{\infty}(P_1)} \leq \mu_*, 
\]
In view of the scaling $D_x^k\partial^{\ell}_t w = \rho_0^{2-k-2\ell} D^k_z \partial^{\ell}_s u$, the derivative estimates for $w$ will follow upon proving that $u \in C^{2,\alpha}(P_{1/2})$ satisfies $\| u - \varphi_1 \|_{C^2(P_{1/2})} \leq \delta$.

We now prove that the conditions $\mathbf{H}_{\varphi}1 ) - \mathbf{H}_{\varphi} 5)$ of the small perturbation Theorem~\cite{small-perturbation} are satisfied in our situation.
    For elements of the class $\cU_{\delta}(\varphi_1)$ and $t \leq 0$, the function $\xi + \varphi_1(s)$ satisfies $1 - \delta \leq \xi + \varphi_1(s) \leq 2m$, so the denominator $4 (\xi + \varphi_1) + |p|^2$ of~\eqref{eqn:F-tilde-operator} is uniformly positive and bounded for $\delta < \frac{1}{2}$.
    For the properties $\mathbf{H}_{\varphi}1 )$ and $\mathbf{H}_{\varphi}2 )$, we use $0 \leq \tilde{\zeta} \leq 1$ to compute, for any $N \geq 0$,
    \begin{align*}
        \left| F(M+N, p , \xi, z,s) - F(M,p,\xi,z,s) - \on{tr} N \right| &= \left| \tilde{\zeta} \tfrac{\la Np, p \rg}{4 \xi + |p|^2} \right| \leq 
        \on{tr} N \tfrac{|p|^2}{4 \xi + |p|^2} \leq \on{tr} N \tfrac{\delta^2}{ 2+ \delta^2}
    \end{align*}
    since $|p| < \delta$ and $\xi > \frac{1}{2}$ in $\cU_{\delta}(\varphi_1)$.
    Bounding $\| N \| \leq \on{tr} N \leq n \| N \|$ for $N \geq 0$, we conclude from this that $F(\cdot,p,\xi,z,s)$ is uniformly elliptic in $\cU_{\delta}(\varphi_1)$.
    Moreover,~\eqref{eqn:F-tilde-operator} implies $\nabla_M F = I - \tilde{\zeta}(z,s) \frac{p \otimes p}{4 \xi + |p|^2}$ independently of $M$, hence $\nabla_M F(A,p,\xi,z,s) - \nabla_M F(B,p,\xi,z,s) = 0$ and we can take the modulus $\omega \equiv 0$ in $\mathbf{H}_{\varphi}5)$.
    The property $\mathbf{H}_{\varphi}3)$ follows from $\cM_{\zeta} \varphi_1 = 0$; to see $\mathbf{H}_{\varphi}4)$, we use $0 \leq \tilde{\zeta} \leq 1$ to bound $\| \nabla_M F \| \leq 1 + \frac{|p|^2}{4 \xi + |p|^2} \leq 1 + \frac{\delta^2}{2}$.
    For $\| M \|, |p| \leq \delta$, we have $|\la Mp, p \rg+ 2|p|^2| \leq (\delta+2) \delta^2$, so
    \begin{align*}
        \partial_{\xi} F &= 4 \tilde{\zeta}(z,s) \frac{\la Mp, p \rg + 2 |p|^2}{(4 \xi + |p|^2)^2}, \qquad \nabla_p F = - \tilde{\zeta}(z,s) \frac{(2 Mp + 4p) (4 \xi + |p|^2) - 2 ( \la Mp, p \rg + 2 |p|^2 )p}{(4 \xi + |p|^2)^2}, \\
        & \implies |\partial_{\xi} F| \leq \tfrac{4 (\delta+2) \delta^2}{16 (1-\delta)^2} < 3 \delta^2, \qquad |\nabla_p F| \leq \tfrac{2(\delta+2) \delta}{4(1-\delta)} + \tfrac{(\delta+2) \delta^3}{8(1-\delta)^2} \leq 6\delta .
    \end{align*}
    Finally, the only dependence of $F$ on the variables $(z,s)$ is through the function $\tilde{\zeta}$, hence $D^k_z \partial^{\ell}_s F = - (D^k_z \partial^{\ell}_s \tilde{\zeta}) \frac{\la Mp, p \rg + 2 |p|^2}{4 \xi + |p|^2}$ where $\left| \frac{\la Mp, p \rg + 2 |p|^2}{4 \xi + |p|^2} \right| \leq 3 \delta^2$ in $\cU_{\delta}(\varphi_1)$.
    For brevity, let $\mu(z,s) := \frac{t_0 + \rho_0^2 s}{h(x_0 + \rho_0 z)^4}$, so that $\tilde{\zeta}(z,s) := \eta(\mu(z,s))$ and $|\mu| \leq 1$.
    We have $\partial_s^{\ell} \mu = \rho_0^2 h^{-4}$ for $\ell=1$ and zero otherwise; moreover, since $h = \tilde{h}^4$ in Definition~\ref{def:k-cutoff-h}, we compute that
    \[
    D_z \mu = \rho_0 \mu \frac{\nabla (h^{-4})}{h^{-4}}(x_0 + \rho_0 z), \quad \frac{D^k \mu}{\rho_0^k \mu} = \sum_{j=1}^k \sum_{|\alpha|=k} c_{\alpha} \frac{D^{\alpha}h}{h^j} (x_0 + \rho_0z), \quad \sum_{k=0}^3 |h|^{\frac{k}{4}}\frac{|D^k_z \mu|}{\rho_0^k \mu} \leq 100 \tau,
    \]
    and $|D^k_z \mu| \leq C_k (\rho_0 |h|^{-\frac{1}{4}})^k$, where $D^{\alpha}h = D^{\alpha_1} h \otimes \cdots \otimes D^{\alpha_j} h$ for any $j$-entry multi-index $\alpha = (\alpha_1, \dots, \alpha_j)$.
    Combining this fact with $|\frac{d^k}{d \sigma^k} \eta(\sigma)| \leq C_k$ and $|\mu| \leq 1$ as well as the iterated chain rule, we can replace $\tau$ by $C(m,n) \tau$ for a sufficiently small constant $C(m,n)$ to obtain
    \begin{equation}\label{eqn:bounds-on-zeta}
        \begin{split}
        |D^k_z \partial_s^{\ell} \tilde{\zeta}| \leq C_{k,\ell} \bigl(\rho_0 \, |h|^{- \frac{1}{4}} \bigr)^k (\rho_0^2 \, |h|^{-4})^{\ell} , \qquad \sum_{k=0}^3 (\rho_0 |h|^{- \frac{1}{4}})^{-k} \left( |D^k_z \tilde{\zeta}| + (\rho_0^2 |h|^{-4})^{-\ell} |D^k_z \partial_s^{\ell} \tilde{\zeta}| \right) \leq \tau.
        \end{split}
    \end{equation}
    Consequently, $|\nabla_{(z,s)} F| \leq 4 \delta^2 (1 + \rho_0^2 \inf_{z \in B_1} h(x_0 + \rho_0 z)^{-4})$.
    The assumption $\rho_0 \leq C(m,n) h(x_0)^2$ implies that $\rho_0 \leq \tilde{C}(m,n) \inf_{B_{\rho_0}(x_0)} h(x)^2$ by the bound~\eqref{eqn:r0-inequality} of Lemma~\ref{lemma:geometric-properties-of-h}, whereby
    \[
    \| \nabla F \|_{L^{\infty}(\cU_{\delta}(\varphi_1))} \leq 10 \delta \bigl[ 1 + \rho_0^2 \inf_{z \in B_1} h(x_0 + \rho_0 z)^{-4} \bigr] \leq C(m,n) \delta.
    \]
    This establishes $\mathbf{H}_{\varphi}4)$, so the small perturbation Theorem of~\cite{small-perturbation} applies to produce a $\mu_*(m,n,\alpha,\delta)$ such that $\| u - \varphi_1 \|_{L^{\infty}(P_1)} \leq \mu_*$ implies $\| u - \varphi_1 \|_{C^2(P_{1/2})} \leq \delta$.
    This proves the claim.
\end{proof}

\subsection{Construction of a solution}\label{section:construction-of-a-solution}

We proceed with the proof of Theorem~\ref{thm:solution-of-smcf-with-estimates}.
We will approximate the domain $\{ (x,t): - h(x)^4 < t < 0 \}$ by a sequence of cylindrical domains, which we call a \textit{staircase}.
Since $x \mapsto h(x)^4$ is a smooth map, Sard's theorem ensures that its set of regular values has full measure, so there exists a $M \in ( \sup_U h(x)^4, 2 \sup_U h(x)^4)$ such that the countable set $\{ t_j \}_{j \in \bN_0} = \{ 2^{-(j+1)} M \}_{j \in \bN_0}$ consists of regular values of the function $h^4$.
Let $\hat{h}_d$ be the discretization of the function $h^4$, taking values in the set $\{ 2^{-(j+1)} M \}_{j \in \bN_0}$ and defined by
\begin{equation}\label{eqn:discretization}
    \hat{h}_d(x) = 2^{-(j+1)} M \qquad \text{where } \; 2^{-(j+1)} M \leq h(x)^4 < 2^{-j} M.
\end{equation}
We also define sub-domains of $U$ by $U_j := \{ x \in U : 2^{-(j+1)} M < h(x)^4 \}$, so that $\partial U_j = \{ x \in U : h(x)^4 = t_j \}$ is a smooth submanifold for every $j$ by construction, where $t_j = 2^{-(j+1)}M$.
The definition of $\hat{h}_d$ implies $ \tfrac{1}{2} h(x)^4 \leq \hat{h}_d(x) \leq h(x)^4$ for all $x \in U$.
We consider the discretized domain
\[
W_d := \{ (x,t) : x \in U, \; - \hat{h}_d(x) < t < 0 \} = \bigcup_{j=0}^{\infty} \bigl( U_j \times (-t_j, 0) \bigr), \qquad \bar{W}_d = \bigcup_{j=0}^{\infty} \bar{Q}_j,
\]
for $Q_j := U_j \times ( - t_j, - t_{j+1})$ the steps of the staircase.
Since $\hat{h}_d(x) < h(x)^4$, $W_d$ is an infinite union of parabolic cylinders with $W_d \subset \{ - h(x)^4 < t < 0 \}$.
The steps $Q_j$ have the following description.
\begin{lemma}\label{lemma:Qj-geometry}
Let $\tilde{W} := \{ (x,t) : - \frac{1}{5} h(x)^4 < t < 0 \}$.
We can take $\tau \in (0,\tau_0(m,n)]$ small so that
    \begin{equation}\label{eqn:qj-intersection}
\tilde{W} \cap Q_{j+1} \subset ( U_j \setminus \tilde{U}_j) \times [ - t_{j+1} , - t_{j+2} ], \qquad \tilde{U}_j := \{ x \in U_j : \textup{dist}(x, \partial U_j) < \tau^{-1} t_j^{\frac{1}{16}} \}
\end{equation}
for every $j$.
In particular, we have $\textup{dist} ( \tilde{W} \cap Q_{j+1} , \partial U_j \times [ -t_{j+1}, - t_{j+2}]) > \tau^{-1} t_j^{\frac{1}{16}}$.
\end{lemma}
\begin{proof}
Consider any $(x,t) \in Q_{j+1} \cap \{ t > - \frac{1}{5} h(x)^4 \}$ and some $x_0 \in \partial U_j$ with $|x - x_0| = \rho_0$.
    Then, $\frac{1}{5} h(x)^4 > |t| \geq t_{j+1} = \frac{1}{2} t_j = \frac{1}{2} h(x_0)^4$, and since $h^{\frac{1}{4}}$ is $\tau$-Lipschitz, we find $x \in U_j$ and
    \[
    h(x)^{\frac{1}{4}} > \bigl(\tfrac{5}{2} \bigr)^{\frac{1}{16}} h(x_0)^{\frac{1}{4}} \implies \Bigl[ \bigl(\tfrac{5}{2} \bigr)^{\frac{1}{16}} - 1 \Bigr] h(x_0)^{\frac{1}{4}} < |h(x)^{\frac{1}{4}} - h(x_0)^{\frac{1}{4}}| < \tau |x - x_0| = \tau \rho_0.
    \]
    Replacing $\tau$ by $10^{-3} \tau$ in the last step, we obtain $\rho_0 = |x - x_0| > \tau^{-1} t_j^{\frac{1}{16}}$ for any such $x_0$, hence $\text{dist}(x, \partial U_j) > \tau^{-1} t_j^{\frac{1}{16}}$.
    Consequently, $x \in U_j \setminus \tilde{U}_j$ and the claimed containment~\eqref{eqn:qj-intersection} is proved.
\end{proof}
We now utilize the above ingredients to prove the existence of the desired solution $w_{\tau}$.
Let $\eta : (-\infty,0] \to [0,1]$ be a smooth cutoff function satisfying $\bigl| \frac{d^k}{d \sigma^k} \eta(\sigma) \bigr| \leq C_k$ and
\[
\eta \equiv 0 \quad \text{on } \; \bigl( - \infty, - \tfrac{1}{60m} \bigr], \qquad \eta \equiv 1 \quad \text{on } \; \bigl[ -\tfrac{1}{80 m}, 0], \qquad 0 \leq |\eta'| \leq 10^3 m.
\]
We define the function $\zeta(x,t) := \eta \bigl( \frac{t}{h(x)^4} \bigr)$ and consider the operator $\cM_{\zeta}$ as in~\eqref{eqn:cmcf-for-u2}, which by construction agrees with the mean curvature flow operator $\cM_1$ on $\{ - \frac{1}{100 m} h(x)^4 < t < 0 \}$, and becomes the linear operator $\cM_0 w = \partial_t w - \Delta w + 2(m-1)$ on $\{ t < - \frac{1}{50 m} h(x)^4 \}$.
Let $w_{\pm}$ denote the barrier functions constructed in Corollary~\ref{cor:general-barrier-functions}, which satisfy $\cM_{\zeta}(w_-) < 0 < \cM_{\zeta}(w_+)$.
On each $Q_j= U_j \times (- t_j, - t_{j+1})$, we will construct a solution of $\cM_{\zeta}w_j =0$ using an inductive procedure.
While these solutions do not glue together to a global solution of $\cM_{\zeta}$ on all of $W_d$, we will prove that they assemble into a smooth solution $w_{\tau}$ on the sub-domain $\tilde{W} := \{ (x,t) : - \frac{1}{5} h(x)^4 < t < 0 \}$ of $W_d$ with $W \subset \tilde{W}$.
The construction of the solution $w_{\tau}$ on $\tilde{W}$ utilizes the following inductive procedure:
\begin{enumerate}[(I)]
    \item On the domain $Q_j$, we solve the initial-boundary value problem for $\cM_{\zeta}$ with data $\phi_j$,
    \begin{align}
        \cM_{\zeta} w_j &= 0 \quad \text{in } \; Q_j, \qquad w_j = \phi_j \quad \text{on } \; \partial_p Q_j,  \qquad \text{where}\label{eqn:w=phij-ibvp} \\
        &w_- \leq \phi_j \leq w_+ \quad \text{on } \; \partial_p Q_j = ( \partial U_j \times [-t_j, - t_{j+1}]) \cup U_j \times \{ -t_j \}. \label{eqn:phi-j-data}
    \end{align}
    \item Using the terminal datum $w_j$ on $U_j \times \{ - t_{j+1} \}$ and the function $w_+$, we interpolate smoothly to produce new initial data $\phi_{j+1}$ on the time-slice $U_{j+1} \times \{ -  t_{j+1} \}$.
    \item We solve the initial-boundary value problem on $Q_{j+1}$ and iterate the process inductively in $j$.
\end{enumerate}
This construction will involve sub-regions of the form $\{ (x,t) : - c_i h(x)^4 < t < 0\}$ for constants $c_1 = \frac{1}{100m} < c_L < \hat{c}_L < c_{\tilde{W}} = \frac{1}{5} < c_p < c_+$ that play the following roles:
\begin{enumerate}[(i)]
    \item In the region $W = \{ - c_1 h(x)^4 < t < 0 \}$, we have the mean curvature flow operator $\cM_{\zeta} = \cM_1$.
    \item Outside the region $\{ - c_L h(x)^4 < t <  0 \}$, we have the linear operator $\cM_{\zeta} = \cM_0$.
    \item The region $\{ - \hat{c}_L h(x)^4 < t < 0 \}$ satisfies Lemma~\ref{lemma:geometric-properties-of-h} for $4 m \hat{c}_L < c_{\tilde{W}}$, namely 
    \[
    (x_0,t_0) \in \{ - \hat{c}_L h(x)^4 < t < 0 \}, \; \rho_0 \leq \sqrt{2 m \hat{c}_L} h(x_0)^2 \implies P_{\rho_0}(x_0,t_0) \subset \{ - c_{\tilde{W}} h(x)^4 < t < 0 \} = \tilde{W}.
    \]
    \item The constants $c_p < c_+ < \frac{1}{2}$ will be used to produce the new initial datum $\phi_{j+1}$ on $U_{j+1} \times \{ - t_{j+1} \}$ by smoothly gluing the terminal datum $w_j$ to the function $w_+$ in a tubular neighborhood of $\partial U_j$ along the time-slice $U_j \times \{ - t_{j+1} \}$.
    Moreover, $c_p > c_{\tilde{W}}$ will preserve the smoothness in $\tilde{W}$.
\end{enumerate}
We observe that an admissible choice of constants satisfying the above requirements is
\[
c_1 = \tfrac{1}{100 m}, \quad c_L = \tfrac{1}{50 m}, \quad \hat{c}_L = \tfrac{1}{30 m}, \quad c_{\tilde{W}} = \tfrac{1}{5}, \quad c_p = \tfrac{2}{5}, \quad c_+ = \tfrac{9}{20},
\]
and will use those values in our construction, keeping in mind the properties (i) -- (iv).

We now implement these steps to construct the desired solution $w_{\tau}$ in $\{ - \frac{1}{5} h(x)^4 < t < 0 \}$.
\begin{proof}[Proof of Theorem~\ref{thm:solution-of-smcf-with-estimates}]
We follow the steps (I)--(III) of the above inductive scheme and show that this procedure applies for each $j$, producing a smooth solution $w_{\tau}$ of $\cM_{\zeta}$ on $\tilde{W}$ via $w_{\tau} = w_j$ on $\tilde{W} \cap Q_j$.

\smallskip \noindent \textbf{Induction statement.}
We will prove the following statement by strong induction: for every $J$, we can find a smooth solution $w_J$ of the problem~\eqref{eqn:w=phij-ibvp} on the domain $Q_J$ such that defining
\[
w_{\tau} = w_j \quad \text{on } \; \tilde{W} \cap \bar{Q}_j, \qquad 0 \leq j \leq J, \qquad \tilde{W} := \{ (x,t) : x \in U, \; - \tfrac{1}{5} h(x)^4 < t < 0 \}
\]
produces a smooth solution of $\cM_{\zeta}$ on the domain $\tilde{W} \cap \bigcup_{j=0}^J Q_j$, which satisfies $w_- \leq w_{\tau} \leq w_+$ and
\begin{equation}\label{eqn:w-tau-final-condition}
\begin{split}
    w_{\tau}^{-1}|\nabla_x w_{\tau}|^2 + |D^2_x w_{\tau}| &\leq \, \delta \qquad \text{and} \qquad \partial_t w_{\tau} < - 1 \qquad \text{in } \; \bigl\{ - \tfrac{1}{30 m} h(x)^4 < t < 0 \}.
\end{split}
\end{equation}
For any $(x_0,t_0) \in W_d$ with $t_0 < - \frac{1}{50 m} h(x_0)^4$, we can find $\rho_0 > 0$ such that $t < - \frac{1}{50 m} h(x)^4$ holds for all $(x,t) \in P_{\rho_0}(x_0,t_0)$, hence $\cM_{\zeta} = \cM_0$ is a linear operator and $w$ is the solution of the linear equation $\cM_0 w = 0$ in $P_{\rho_0}(x_0,t_0)$.
Thus, $w$ cannot have gradient blowup in this parabolic neighborhood, and covering $\{ t < - \frac{1}{50 m} h(x)^4 \}$ by such parabolic cylinders proves that $|\nabla w|$ is uniformly bounded in terms of the data $\phi$ and the geometry of the region.
In particular, if $t < - \frac{1}{50 m} h(x)^4$ holds for all $(x,t) \in Q_j$, then the problem~\eqref{eqn:w=phij-ibvp} has a solution $w \in C^{\infty}(Q_j) \cap C^0(\bar{Q}_j)$.

For the base case of the induction, observe that $t \leq - t_1 = - \frac{M}{4} < - \frac{1}{10m} \sup_U h(x)^4$ on $Q_0$, hence $\cM_{\zeta} = \cM_0$ and our observation implies the existence of a solution $w_0 \in C^{\infty}(Q_0) \cap C^0(\bar{Q}_0)$ to~\eqref{eqn:w=phij-ibvp} with data $\phi_0 = w_+|_{\partial_p Q_0}$.
Defining $w_{\tau} = w_0$ on $\tilde{W} \cap Q_0$ proves the case $j=0$.

\smallskip \noindent \textbf{Step 1: The initial-boundary value problem.}
We now suppose that the solution $w_{\tau}$ is defined on the domain $\tilde{W} \cap \bigcup_{j=0}^J Q_j$ for some $J$, and will show that it can be extended to $\tilde{W} \cap \Bigl( \bigcup_{j=0}^J Q_j \cup Q_{J+1} \Bigr)$.
For the inductive step, let $w_J$ be the solution of~\eqref{eqn:w=phij-ibvp} on $Q_J$ with boundary datum $\phi_J$ satisfying~\eqref{eqn:phi-j-data}, hence $w_- \leq w_J \leq w_+$ by the comparison principle of Lemma~\ref{lemma:comparison-principle}.
We fix a $C^{\infty}$ cutoff function $\chi$ with
\[
\chi: \bR \to [0,1], \qquad \chi \equiv 1 \quad \text{on } \; \bigl[0, \tfrac{2}{5} \bigr], \qquad \chi \equiv 0 \; \text{on } \; \bigl[ \tfrac{9}{20} ,\infty \bigr), \qquad 0 \leq |\chi'| \leq 10.
\]
Then, we define the $(J+1)$-st initial datum on $\partial_p Q_{J+1}$ by setting
\[
\phi_{J+1}(x) := \Bigl[ 1 - \chi \Bigl( \tfrac{t_{J+1}}{h(x)^4} \Bigr) \Bigr] w_+ (x , - t_{J+1}) + \chi \Bigl( \tfrac{t_{J+1}}{h(x)^4} \Bigr) w_J (x , - t_{J+1}) \quad \text{on } \; U_J \times \{ - t_{J+1} \},
\]
and $\phi_{J+1} = w_+$ on $\partial U_{J+1} \times [ - t_{J+1}, - t_{J+2}] \cup (U_{J+1} \setminus U_J) \times \{ - t_{J+1} \}$.
Since $\chi \equiv 0$ on $[ \frac{9}{20} , \infty)$ and $\chi \equiv 1$ on $[0, \frac{2}{5}]$, we have $\chi \bigl( \frac{t_{J+1}}{h(x)^4} \bigr) = 0$ for $x \in U_J$ with $h(x)^4 \leq \tfrac{10}{9} t_J$ and $1 - \chi \bigl( \frac{t_{J+1}}{h(x)^4} \bigr) = 0$ for $h(x)^4 \geq \tfrac{5}{4} t_J$.
Therefore, $\phi_{J+1} = w_J$ for $h(x)^4 \geq \tfrac{5}{4} t_J$ and $\phi_{J+1}$ extends smoothly to $\phi_{J+1}(x) = w_+(x, - t_{J+1})$ for $x \in U_{J+1} \cap \{ h(x)^4 \leq \frac{20}{9} t_{J+1}\}$.
In particular, $\phi_{J+1} \in C^{\infty} (\partial_p Q_{J+1})$ and $w_- \leq \phi_{J+1} \leq w_+$.
We let
\[
    Q^{(\sigma)}_{J+1} := U_{J+1} \times ( - t_{J+1}, \sigma - t_{J+1}) \subseteq Q_{J+1}, \qquad \sigma \in (0, t_{J+1} - t_{J+2}].
\]
By construction, the domain $U_{J+1}$ has smooth boundary and the datum $\phi_{J+1} \in C^{\infty} (\partial_p Q_{j+1})$ is smooth; in particular, it satisfies the $C^0$ compatibility condition at the corner
$\partial U_{J+1} \times \{ -t_{J+1} \}$ as in~\cite{ladyzhenskaya}*{Ch. V \S 6}.
The property~\eqref{eqn:qj-intersection} implies $\tilde{W} \cap \{ t = - t_{J+1} \} \subset (U_{J+1} \setminus \tilde{U}_{J+1}) \times \{ - t_{J+1} \}$, so the operator $\cM_{\zeta} = \cM_0$ is linear in a neighborhood of the corner $\partial U_{J+1} \times \{ - t_{J+1} \}$.
Linearizing $\cM_{\zeta}$ at $\phi_{J+1}$ produces a strictly parabolic operator $(\partial_t - \cL_{\zeta})$, so we can apply the linear theory for the local Dirichlet problem with zeroth-order compatibility at the corner~\cite{ladyzhenskaya}*{Ch. IV \S 5} to deduce invertibility in the Banach space $C^{2+\alpha, 1 + \frac{\alpha}{2}}$.
Using the Schauder fixed point theorem as in~\cite{lieberman}*{Theorem 2.2} and the local existence results~\cite{lieberman-book}*{Theorems 8.2 \& 8.3}, we obtain a $C^{2+\alpha, 1 + \frac{\alpha}{2}}$ solution $w_{J+1}$ on $Q^{(\ve)}_{J+1}$ for small $\ve > 0$, such that for any $V \subset U_{J+1}$ with $\textup{dist}( \partial V, \partial U_{J+1} ) > 0$, we have
\begin{equation}\label{eqn:w-J+1-ibvp}
w_{J+1} \in C^{2+\alpha, 1 + \frac{\alpha}{2}} (Q_{J+1}^{(\ve)}) \cap C^0( \bar{Q}^{(\ve)}_{J+1}), \qquad w_{J+1} \in C^{2+\alpha, 1 + \frac{\alpha}{2}} ( V \times [ - t_{J+1}, \ve - t_{J+1} ]).
\end{equation}
See Remark~\ref{rmk:solve-ibvp-by-approximation} for an alternative proof of the short-time existence of a solution using an approximation by data $\phi^{(k)}_{J+1}$ satisfying the compatibility conditions~\cite{ladyzhenskaya}*{Ch. V \S 6} of all orders at the corner.

Let $\cS$ denote the set of $\sigma \in (0,t_{J+1} - t_{J+2}]$ such that the problem~\eqref{eqn:w=phij-ibvp} admits a solution $w_{J+1}$ on $Q^{(\sigma)}_{J+1}$ with finite gradient.
For any $\sigma \in \cS$, once $w_{J+1}$ satisfies the properties~\eqref{eqn:w-J+1-ibvp}, the equation is classically parabolic and $\phi_{J+1} \in C^{\infty} (\partial_p Q_{J+1})$, so the Krylov-Safonov theorem and standard bootstrapping arguments by parabolic Schauder theory of~\cite{ladyzhenskaya} give $w_{J+1} \in C^{\infty}_{\text{loc}}(Q^{(\sigma)}_{J+1})$, and $w_{J+1} \in C^{\infty}( V \times [-t_{J+1}, \sigma - t_{J+1}])$ is smooth up to the initial time-slice whenever $\text{dist}( \partial V, \partial U_{J+1}) > 0$, by the initial-boundary regularity theory of Lieberman~\cite{initial-boundary-regularity}.
Thus, by a continuity argument, either $\cS$ is open and closed, so $\cS \neq \varnothing$ makes $\cS = (0,t_{J+1} - t_{J+2}]$, or there exists a maximal $\sigma^* < (t_{J+1} - t_{J+2})$ such that $\sup_{Q_{J+1}^{(\sigma)}} |\nabla w_{J+1}| \to \infty$ as $\sigma \uparrow \sigma^*$.
Since $w_- \leq \phi_{J+1} \leq w_+$, the comparison principle~\ref{lemma:comparison-principle} gives $w_- \leq w_{J+1} \leq w_+$ for the maximal solution $w_{J+1} \in C^{\infty}(Q^{(\sigma^*)}_{J+1}) \cap C^0( \bar{Q}^{(\sigma^*)}_{J+1})$ in this domain.

\smallskip \noindent \textbf{Step 2: Extending the solution $w_{\tau}$ to $Q_{J+1}^{(\sigma^*)}$.}
We claim that extending $w_{\tau}$ to $\tilde{W} \cap \left( \bigcup_{j=0}^J Q_j \cup Q_{J+1}^{(\sigma^*)} \right)$ by $w_{\tau} = w_{J+1}$ on $Q_{J+1}^{(\sigma^*)}$ produces a smooth solution of $\cM_{\zeta}$.
Since $w_{\tau}$ is smooth on $\tilde{W} \cap \bigcup_{j=0}^J Q_j$, by the inductive hypothesis, and on $Q^{(\sigma^*)}_{J+1}$, by the discussion of Step 1, it suffices to examine the interface $\tilde{W} \cap \bigl( \bar{Q}_J \cap \bar{Q}_{J+1}) =  \tilde{W} \cap \bigl( \bar{U}_J \times \{ - t_{J+1} \} )$.
Using $\phi_{J+1} = w_J$ for $h(x)^4 \geq \frac{5}{4} t_J$, we show that the functions $w_J$ and $w_{J+1}$ glue smoothly across the interface $\{ h(x)^4 > \frac{5}{4} t_J \} \times \{ - t_{J+1} \}$, so that
\[
w_{\tau} = w_J \; \text{ on } \; \{ h(x)^4 > \tfrac{5}{4} t_J \} \times (-t_J, -t_{J+1}], \quad w_{\tau} = w_{J+1} \; \text{ on } \; \{ h(x)^4 > \tfrac{5}{4} t_J \} \times (- t_{J+1}, \sigma^* - t_{J+1}],
\]
defines a smooth function.
Indeed, consider a point $(x_0,- t_{J+1})$ on the interface and some $\rho_0 > 0$ with 
\[
P^{\pm}_{2\rho_0}(x_0,- t_{J+1}) := B_{2\rho_0}(x_0) \times [-t_{J+1} - (2\rho_0)^2, - t_{J+1} + (2\rho_0)^2] \subset \{ h(x)^4 > \tfrac{5}{4} t_J \} \times [ - t_J , \ve - t_{J+1} ].
\]
In particular, $\phi_{J+1} = w_J$ on $B_{2\rho_0}(x_0) \times \{ - t_{J+1} \}$ and the operator $\cM_{\zeta}$ is uniformly parabolic in $P^{\pm}_{2\rho_0}(x_0, - t_{J+1})$, hence~\cite{initial-boundary-regularity}*{Theorem 1.5} implies that $w_{J+1} \in C^{\infty} (B_{\rho_0}(x_0) \times [-t_{J+1}, - t_{J+1} + \rho_0^2])$.
Moreover, $w_J \in C^{\infty} (B_{\rho_0}(x_0) \times [ -t_{J+1}, - t_{J+1} + \rho_0^2])$ by the inductive hypothesis.
Consequently, all spatial derivatives agree on the interface, meaning that
\[
D^k_x w_{J+1}(x, - t_{J+1}) = D^k_x w_J(x, - t_{J+1}) \quad \text{on } \; B_{\rho_0}(x_0) \times \{ - t_{J+1} \} .
\]
By smoothness, evaluating the PDE $\cM_{\zeta} w = \partial_t w - F(x,t,w,\nabla w, \nabla^2 w) = 0$ at $(x, - t_{J+1})$ gives
\[
\partial_t w_{J+1}(x, - t_{J+1}) = F(x, - t_{J+1}, w_J, \nabla w_J, D^2 w_J) = \partial_t w_J (x, - t_{J+1})
\]
and hence the glued function $w_{\tau}$ is $C^{2,1}$ across the interface $B_{\rho_0}(x_0) \times \{ - t_{J+1} \}$.
Finally, standard bootstrapping arguments by parabolic Schauder theory as in~\cite{ladyzhenskaya} show that
\[
w_{\tau} \in C^{2,1} ( B_{\rho_0}(x_0) \times [ - t_{J+1} - \rho_0^2, - t_{J+1} + \rho_0^2] ) \implies w \in C^{\infty} \Bigl( B_{\frac{\rho_0}{2}} (x_0) \times \bigl[ - t_{J+1} - \bigl( \tfrac{\rho_0}{2} \bigr)^2, - t_{J+1} + \bigl( \tfrac{\rho_0}{2} \bigr)^2 \bigr] \Bigr)
\]
so the glued function $w_{\tau}$ is smooth.
Next, Lemma~\ref{lemma:Qj-geometry} implies that $\tilde{W} \cap ( \bar{U}_J \times \{- t_{J+1} \} )$ remains a positive distance away from $\partial U_J \times \{ - t_{J+1} \}$.
Thus, for any $(x_0, - t_{J+1}) \in \tilde{W}$, we can find some sufficiently small $\rho_0 < \frac{1}{10} \min \{ \sqrt{\sigma^*} , h(x_0)^4, t_{J+1}^{10} \}$ to obtain a full parabolic neighborhood
\[
B_{\rho_0}(x_0) \times [ - t_{J+1} - \rho_0^2 , - t_{J+1} + \rho_0^2] \subset \tilde{W} \cap \bigl( \{ h(x)^4 > \tfrac{5}{4} t_J \} \times \{ - t_J, \sigma^* - t_{J+1} \} \bigr)
\]
where $w_{J+1}$ glues smoothly with $w_J$, thereby extending $w_{\tau}$ to $\tilde{W} \cap \Bigl( \bigcup_{j=0}^J Q_j \cup Q_{J+1}^{(\sigma^*)} \Bigr)$ as claimed.
Moreover, $w_- \leq w_{\tau} \leq w_+$ in $\tilde{W} \cap \bigcup_{j=0}^J Q_j$, by the inductive hypothesis, and $w_- \leq w_{J+1} \leq w_+$ by the comparison principle above, therefore $w_- \leq w_{\tau} \leq w_+$ in $\tilde{W} \cap \left( \bigcup_{j=0}^J Q_j \cup Q_{J+1}^{(\sigma^*)} \right)$.

\smallskip \noindent \textbf{Step 3: Derivative bounds for $w_{\tau}$ in $\tilde{W}$.}
Our initial observation shows that any gradient blowup of $w_{J+1}$ cannot occur in the region $\{ t < - \tfrac{1}{50 m} h(x)^4 \}$ where $\cM_{\zeta}$ coincides with the linear operator $\cM_0$, and Lemma~\ref{lemma:Qj-geometry} shows that $Q_{J+1} \cap \{ t \geq - \frac{1}{30 m }h(x)^4 \} \subset (U_J \setminus \tilde{U}_j) \times [ -t_{J+1}, - t_{J+2}]$ remains a positive distance away from $\partial U_{J+1} \times [ - t_{J+1} - t_{J+2}]$.
Thus, it suffices to prove that $w_{J+1}$ cannot have gradient blowup at any $(x_0, t_0) \in (U_J \setminus \tilde{U}_J) \times [ - t_{J+1}, \sigma^* - t_{J+1}]$ with $t_0 \geq - \frac{1}{30 m} h(x_0)^4$ and $h(x_0)^4 > t_{J+1}$.
Such points lie in the region $(x_0,t_0) \in \tilde{W} = \{ - \frac{1}{5} h(x)^4 < t < 0 \}$ where $w_{J+1}$ agrees with the glued solution $w_{\tau}$, thus we can replace $w_{J+1}$ by $w_{\tau}$ in the ensuing arguments.

Using $|t_0| \leq \frac{1}{30 m} h(x_0)^4$, the construction of $w_{\pm}$ from Corollary~\ref{cor:general-barrier-functions} shows that
\begin{align*}
w_{\tau}(x_0,t_0) &< w_+(x_0,t_0) < 2(m-1) |t_0| + 2 h(x_0)^3 \exp \bigl( - h(x_0)^{- \frac{9}{8}}) < \tfrac{1}{15} h(x_0)^4
\end{align*}
for $|h|_{L^{\infty}(U)} \leq \tau_0$ sufficiently small.
Taking $\rho_0 = \sqrt{w_{\tau}(x_0,t_0)}$, we therefore obtain $\rho_0 < \sqrt{2mc} \, h(x_0)^2$ with $c = \frac{1}{30m}$ and $4cm < \frac{1}{5}$.
Now, Lemma~\ref{lemma:geometric-properties-of-h} implies that $P_{\rho_0}(x_0, t_0) \subset \tilde{W}$, hence $w_{\tau}$ is defined in the backwards parabolic cylinder $P_{\rho_0}(x_0,t_0)$.
Since $\rho_0 < h(x_0)^{\frac{9}{5}}$ and $P_{\rho_0}(x_0,t_0) \subset \tilde{W}$, we can apply the estimate~\eqref{eqn:u-plus-u-minus-refined-inequality} from Lemma~\ref{lemma:tight-barrier-estimates} with $\varphi_{w_{\tau}(x_0,t_0)}(t) = w_{\tau}(x_0,t_0) - 2(m-1)t$ the cylinder solution, which has $\varphi_{w_{\tau}(x_0,t_0)}(t) < 2m \rho_0^2 = 2 m w_{\tau}(x_0,t_0)$ in $P_{\rho_0}(x_0,t_0)$.
We therefore obtain
\begin{equation}\label{eqn:w-tau-two-bounds}
    \begin{split}
    \sup \bigl| w_{\tau}(x,t) - \varphi_{w_{\tau}(x_0,t_0)}(t - t_0) \bigr| \leq C(m,n) h(x_0)^{\frac{1}{10}} \rho_0^2 \qquad \text{in } \; P_{\rho_0}(x_0,t_0).
    \end{split}
\end{equation}
We may then take $\tau_0(m,n,\delta)$ sufficiently small and use $|h|_{L^{\infty}(U)} < \tau_0$ to make $C(m,n) h(x_0)^{\frac{1}{10}} < \mu_*$, for $\mu_*(m,n,\delta^2)$ defined in Proposition~\ref{prop:apply-small-perturbation} with $\alpha = \frac{1}{2}$ and $\delta$ replaced by $\delta^2$.
Therefore, the small perturbation Theorem applies to conclude that $w_{\tau} \in C^{2,\alpha} (P_{\rho_0/2}(x_0,t_0))$ and
\begin{equation}\label{eqn:w-full-hessian-bound}
\sup_{P_{\rho_0/2}(x_0,t_0)} \left( \rho_0^{-2} | w_{\tau} - \varphi_{\rho_0^2}(t-t_0) | + \rho_0^{-1} |\nabla_x w_{\tau}| + |D^2_x w_{\tau}| + |\partial_t w_{\tau} + 2(m-1)| \right) \leq \delta^2.
\end{equation}
In particular, $|\nabla_x w_{\tau}|^2 \leq \delta w_{\tau} < 4mh(x)^4$, so $w_{\tau}$ cannot have gradient blowup in $P_{\rho_0/2}(x_0,t_0)$.
Covering the region $\{ t > - \frac{1}{30m} h(x)^4 \}$ by such parabolic cylinders therefore shows that $w_{J+1}$ cannot have gradient blowup in any sub-domain $Q_{J+1}^{(\sigma)} \subseteq Q_{J+1}$, so it extends to a solution on all of $Q_{J+1}$.

\smallskip \noindent \textbf{Step 4: Conclusion.}
The above discussion shows that letting $w_{\tau} = w_{J+1}$ on $\tilde{W} \cap \bar{Q}_{J+1}$ produces a smooth function on $\tilde{W} \cap \bigcup_{j=0}^{J+1} Q_j$ with $w_- \leq w_{\tau} \leq w_+$, which satisfies $|D^2_x w_{\tau}| \leq \delta^2$ and $\partial_t w_{\tau} < -1$ in $\{ - \frac{1}{30 m} h(x)^4 < t < 0\}$ by~\eqref{eqn:w-full-hessian-bound}.
Moreover, for $(x,t) \in P_{\rho_0}(x_0,t_0)$, the bound~\eqref{eqn:w-tau-two-bounds} implies that $\frac{1}{2} w_{\tau}(x_0,t_0) < w_{\tau}(x,t) < 4m w_{\tau}(x_0,t_0)$ for $|h|_{L^{\infty}(U)} \leq \tau_0$ sufficiently small, so~\eqref{eqn:w-full-hessian-bound} gives $|\nabla_x w_{\tau}(x,t)|^2 < 2 \delta^3 w_{\tau}(x,t)$ for $\delta \in (0, \frac{1}{2m})$.
These two estimates show that $w_{\tau}^{-1} |\nabla_x w_{\tau}(x,t)|^2 + |D^2_x w_{\tau}| < \delta$, establishing the properties~\eqref{eqn:w-tau-final-condition} of $w_{\tau}$ and completing the induction.
Since $\zeta = 1$ in $W = \{ - \frac{1}{100m} h(x)^4 < t < 0 \}$, we have $\cM_1 w_{\tau} = 0$ and the bounds $w_- \leq w_{\tau} \leq w_+$ imply $\bigl| h(x)^{-3} \exp( h(x)^{- \frac{9}{8}}) ( w_{\tau} + 2(m-1) t ) - 1 \bigr| < h(x)$ as in Corollary~\ref{cor:general-barrier-functions}.
For $|h|_{L^{\infty}(U)} \leq \tau_0$ sufficiently small, this lets us further bound $\exp( - \frac{1}{h(x)^2}) < w_{\tau} + 2(m-1) t < \exp ( -\frac{1}{h(x)^{\frac{11}{8}}})$, so $w_{\tau}$ satisfies all the desired properties in the domain $W$.
\end{proof}

\begin{remark}\label{rmk:solve-ibvp-by-approximation}
    For an alternative proof of the short-time solvability of the initial-boundary value problem~\eqref{eqn:w=phij-ibvp} in $Q^{(\sigma)}_{J+1}$ with data $\phi_{J+1}$, we could follow the computations of Corollary~\ref{cor:general-barrier-functions} to see that the modified barrier functions $\tilde{w}_{\pm} = w_{\pm} \pm (h(x)^4 + t) \exp ( - h(x)^{- \frac{9}{8}})$ again have $\cM_{\zeta} \tilde{w}_- < 0 < \cM_{\zeta} \tilde{w}_+$ and $\tilde{w}_- < w_- < w_+ < \tilde{w}_+$ everywhere in $\{ - h(x)^4 < t < 0 \}$, for $|h|_{L^{\infty}(U)} \leq \tau_0$ small, and satisfy the estimates of Lemma~\ref{lemma:tight-barrier-estimates}.
We consider a sequence of data $\phi^{(k)}_{J+1} \in C^{\infty} ( \partial_p Q_{J+1}^{(\sigma)})$ such that $\phi^{(k)}_{J+1} \to \phi_{J+1}$ uniformly on $\partial_p Q_{J+1}$, each $\phi^{(k)}_{J+1}$ satisfies the compatibility conditions of all orders at the corner $\partial U_{J+1} \times \{ -t_{j+1} \}$ as in~\cite{ladyzhenskaya}*{Ch. V \S 6}, $\tilde{w}_- \leq \phi^{(k)}_{J+1} \leq \tilde{w}_+$ on $\partial_p Q_{J+1}$, and $\phi^{(k)}_{J+1} = \phi_{J+1} = w_J$ on $\{ h(x)^4 > \tfrac{5}{4} t_J \} \times \{ - t_{J+1} \}$.
The latter construction is possible because Lemma~\ref{lemma:Qj-geometry} implies that
\[
\textup{dist} ( \{ h(x)^4 = \tfrac{5}{4} t_J \}, \{ h(x)^4 = \tfrac{10}{9} t_J \}  ) > 0 \qquad \text{and} \qquad \textup{dist}( \{ h(x)^4 = \tfrac{5}{4} t_J \} , \partial U_j ) > 0.
\]
Applying the discussion of this section to the problem~\eqref{eqn:w=phij-ibvp} for $\phi_{J+1}^{(k)}$, we obtain solutions $w_{J+1}^{(k)} \in C^{\infty} (\bar{Q}^{(\sigma^*)}_{J+1})$ for each $k$.
Finally, the bounds~\eqref{eqn:w-tau-two-bounds} obtained through the inductive step would apply to each $w^{(k)}_{J+1}$, allowing us to take a uniform subsequential limit $w^{(k)}_{J+1} \xrightarrow{C^{\infty}_{\textup{loc}}} w_{j+1}$, by Arzel\`a-Ascoli, with $w_{J+1} = \phi_{j+1}$ on $\partial_p Q^{(\sigma^*)}_{J+1}$ and $w_{J+1} \in C^{\infty}(Q_{J+1}^{(\sigma^*)}) \cap C^0( \bar{Q}^{(\sigma^*)}_{J+1})$ having the claimed regularity.
\end{remark}

\begin{corollary}\label{corollary:higher-derivative}
For each $k, \ell \in \bN_0$, there is a constant $C = C(m, n, k, \ell)$ such that
\begin{equation}\label{eqn:interior-estimate}
    \Bigl| D^k_x \partial_t^{\ell} \bigl( w_{\tau}(x,t) + 2(m-1) t \bigr) \Bigr| \leq C \tau  \exp( - \tfrac{8}{h(x)}) \quad \text{in } \; \bigl\{ -\tfrac{1}{200 m} h(x)^4 < t < - \tfrac{1}{400 m} h(x)^4 \bigr\}.
\end{equation}
\end{corollary}
\begin{proof}
    Since the functions $w_{\tau}$ and $\varphi_0 = - 2(m-1) t$ are both solutions of the equation $\cM_1 w = 0$, their difference $v := w_{\tau} + 2(m-1) t$ satisfies a linear parabolic equation of the form
    \begin{equation}\label{eqn:v-linear-equation}
    \partial_t v - a^{ij}(x,t) v_{ij} - b^k(x,t) v_k - c(x,t) v = 0 \qquad \text{on } \; W := \{ (x,t) : x \in U, \; -h(x)^4 < t < 0 \}.
    \end{equation}
    The coefficients of this operator are obtained by linearizing the equation $\cM_1 w = 0$ along the segment $\tilde{w}_{\sigma} := (1-\sigma) ( - 2(m-1) t) + \sigma w_{\tau}$, for $0 \leq \sigma \leq 1$.
    We write
    \[
    \cM_1(u) = \partial_t u - A^{ij}(u, \nabla u) u_{ij} + 2(m-1) + B(u, \nabla u), \qquad A = I - \frac{p \otimes p}{4 r + |p|^2}, \quad B(r,p) = \frac{2 |p|^2}{4r + |p|^2}.
    \]
    We therefore obtain the coefficients of~\eqref{eqn:v-linear-equation} as
    \begin{align*}
    0 &= \cM_1(w) - \cM_1 \bigl( - 2(m-1) t \bigr) = \int_0^1 \frac{d}{d \sigma} \cM_1 [ \tilde{w}_{\sigma} ] = \partial_t v - a^{ij} v_{ij} - b^k v_k - cv, \\
    a^{ij} &= \delta^{ij} - \int_0^1 \frac{(\tilde{w}_{\sigma})_i ( \tilde{w}_{\sigma})_j}{4 \tilde{w}_{\sigma} + |\nabla \tilde{w}_{\sigma}|^2}, \quad b^k = \int_0^1 \bigl[ \partial_{p_k} A^{ij} (\tilde{w}_{\sigma})_{ij} - \partial_{p_k} B (\tilde{w}_{\sigma})], \quad c = \int_0^1 \frac{4Q(\tilde{w}_{\sigma}) + 8 |\nabla \tilde{w}_{\sigma}|^2}{( 4 \tilde{w}_{\sigma} + |\nabla \tilde{w}_{\sigma}|^2)^2} .
    \end{align*}
    Using the bounds $|\nabla w_{\tau}|^2 \leq \delta w_{\tau}$ and $\exp( - h(x)^{-2}) < w_{\tau} + 2(m-1) t < \exp( - h(x)^{- \frac{11}{8}})$ from Theorem~\ref{thm:solution-of-smcf-with-estimates}, we find $|\nabla \tilde{w}_{\sigma}|^2 = \sigma^2 |\nabla w_{\tau}|^2$ and $w_{\tau} \leq 10^3 m \delta^{-1} \tilde{w}_{\sigma}$ for any $\sigma \in [0,1]$ and $t < - \frac{1}{10^3m} h(x)^4$.
    Consequently, the operator~\eqref{eqn:v-linear-equation} is uniformly parabolic due to the uniform lower bound
    \[
    \lambda_{\min} (a^{ij}) \geq 1 -  \tfrac{|\nabla \tilde{w}_{\sigma}|^2}{4 \tilde{w}_{\sigma} + |\nabla \tilde{w}_{\sigma}|^2} = \tfrac{4}{4 + \tilde{w}_{\sigma}^{-1} |\nabla \tilde{w}_{\sigma}|^2} \geq \tfrac{1}{10^4 m}, \qquad \text{due to } \; |\nabla \tilde{w}_{\sigma}|^2 \leq |\nabla w_{\tau}|^2 \leq \delta w_{\tau} \leq 10^3 m \tilde{w}_{\sigma}.
    \]
    Moreover, the estimates of Theorem~\ref{thm:solution-of-smcf-with-estimates} imply that $v < \exp ( - h(x)^{- \frac{11}{10}})$.
    For any spacetime point $(x_0,t_0)$ with $- \frac{1}{200 m} h(x_0)^4 < t_0 < - \frac{1}{400m} h(x_0)^4$, we set $\rho := \frac{1}{10^3 m} h(x_0)^2$, hence~\eqref{eqn:small-enclosure-h(x0)} implies $\frac{99}{100} h(x_0) \leq h(x) \leq \frac{101}{100} h(x_0)$ for any $x \in B_{\rho}(x_0)$, for $\tau$ chosen sufficiently small.
    Consequently, the parabolic cylinder $P^{\pm}_{\rho}(x_0, t_0) := B_{\rho}(x) \times (t_0 - \rho^2, t_0 + \rho^2)$ satisfies $P^{\pm}_{\rho} \Subset W \cap \{ t < - \frac{1}{10^3 m} h(x)^4 \} $, so the standard interior parabolic Schauder estimates for equation~\eqref{eqn:v-linear-equation} as in \cites{ladyzhenskaya, lieberman-book} give
\[
\| D^k_x \partial_t^{\ell} v \|_{L^{\infty}(P_{\rho/2})} \leq C_{k \ell} \rho^{- (k + 2 \ell)} \|v\|_{L^{\infty}(P_{\rho})} \leq C_{k \ell} h(x)^{- 2 k - 4 \ell} \cdot \exp ( - h(x)^{- \frac{11}{10}}) \leq \tilde{C}_{k \ell} \tau \exp ( - \tfrac{8}{h(x)} ) 
\]
with $C_{k\ell}, \tilde{C}_{k\ell} = C, \tilde{C}(m,n,k,\ell)$ independent of $\tau,  \rho, y, t_0$.
This proves~\eqref{eqn:interior-estimate}.
\end{proof}
\begin{remark}\label{rmk:expand-the-domain}
It is clear that the above constructions can more generally be performed for arbitrary $0 < M_1 < M_2 < M_0$, producing a solution $w_{\tau}$ with $\partial_t w_{\tau} <-1$ and derivative estimates in the domain
\begin{align}
    &W := \Bigl\{ (x,t) \in \bR^n \times (-\infty, 0] \; : \; x \in U, \; - M_0 h(x)^4 < t < 0 \Bigr\}, \qquad U := \bR^n \setminus K, \label{eqn:W-input-M} \\
    & \left| D^k_x \partial_t^{\ell} \bigl( w_{\tau} + 2(m-1) t \bigr) \right| \leq C_{k \ell} \tau \exp \bigl( - \tfrac{8}{h(x)} \bigr) \qquad \text{in } \; \{ - M_2 h(x)^4 < t < - M_1 h(x)^4 \}, \label{eqn:interior-estimates-general}
\end{align}
for $C_{k\ell} = C(m,n,M_i,k,\ell)$, after rescaling $h(x)$ appropriately and taking $\tau_0(\delta,m,n,M_i)$ small.
\end{remark}

\section{The main theorem}\label{section:main-theorem}

We now prove Theorem~\ref{thm:main-theorem}.
We use Theorem~\ref{thm:solution-of-smcf-with-estimates} to produce a family of smooth $O(m)$-invariant hypersurfaces $\cG_t \subset \bR^{m+n}$ evolving by mean curvature flow in a Riemannian metric of the form
\begin{equation}\label{eqn:general-metric}
g_f|_{(x,\xi)} = \sum_{i=1}^n dx_i^2 + f(x,|\xi|^2) \sum_{j=1}^m d \xi_j^2 .
\end{equation}
Here $f(x,r)$ is a positive function, to be constructed, satisfying $\sup |D_x^k \partial^{\ell}_r (f-1)| < C_{k,\ell} \tau$ for all $k,\ell \geq 0$.
Given a family of hypersurfaces $\cG_t \subset \bR^{m+n}$ evolving by mean curvature flow, an \textit{arrival time} is a function $\tau$ defined on the set of reach of the flow, with values in $(-\infty,0]$, such that $\cG_t = \{ y : \tau(y) = t \}$.
In our situation, working with $\cG_t = \{ |\xi|^2 = w(x,t) \}$ leads us to consider a modified arrival time function $T(x,r)$  characterized by $T(x,|\xi|^2) =T(x,w(x,t)) = t$, so that $w(x,T(x,r)) = r$.
The two notions of arrival time are related by $\tau(x,\xi) = T(x, |\xi|^2)$ on the set of reach of the flow, and the cylinder solution $\varphi_0(t)$ has $T_0(r) = - \frac{r}{2(m-1)}$.
\begin{lemma}\label{lemma:arrival-time-property}
    The function $T(x,r)$ is an arrival time in the metric $g_f$ if and only if $\cP_f(T) = 0$, where
\begin{equation}\tag{A}\label{eqn:arrival-time-O(m)-invariant}
\begin{split}
  &\cP_f(T) := 1 + \Delta_x T + \frac{2 T_r + 4 r T_{rr}}{f} - f \frac{Q(T) + 8 r\frac{T_r}{f} \la \nabla_x T, \nabla_x T_r \rg + 8 r\frac{T_r^2}{f^2}(T_r + 2 r T_{rr})}{4 r T_r^2 + f \, |\nabla_x T|^2} \\
  & + \left( \frac{m}{2f} + \frac{2 f^{-1} r T_r^2}{4 r T_r^2 + f \, |\nabla_x T|^2} \right) \la \nabla_x f, \nabla_x T \rg + \frac{2r f_r}{f^2} T_r \left( \frac{(m-1)f}{r f_r} + m-2 + \frac{4 r T_r^2}{4 r T_r^2 + f \, |\nabla_x T|^2} \right).
\end{split}
\end{equation}
Here, $Q(T) := T_{x_i} T_{x_j} T_{x_i x_j} = D^2T( \nabla_x T , \nabla_x T)$ and $f, \nabla_x f, f_r$ are evaluated at points $(x,r)|_{r = |\xi|^2}$.
\end{lemma}
\begin{proof}
    The evolving hypersurfaces $\cG_t$ are the level sets $\{ F =0 \}$ of the function $F(x,\xi,t) := T(x,|\xi|^2) - t$, so using $\partial_t F = - 1$ shows that they evolve by mean curvature flow under $g_f$ if and only if
    \begin{equation}\label{eqn:level-set-equation-here}
    1 + \on{div}_g \Bigl( \tfrac{\nabla_g F}{|\nabla_g F|_g} \Bigr) |\nabla_g F|_g = 0
    \end{equation}
    Let $r := |\xi|^2$, so using $\sqrt{\det g} = f^{\frac{m}{2}}$ and denoting $W := |\nabla_g F|^2_g = |\nabla_x T|^2 + \frac{4 r T_r^2}{f}$ produces $\nabla_g F = (\nabla_x T, \frac{2T_r}{f} \xi_j \partial_{\xi_j})$ and $|\nabla_g F|_g = \sqrt{W}$.
    Using $\on{div}_g X = f^{- \frac{m}{2}} \partial_a ( f^{\frac{m}{2}} X^a)$, we compute that
    \begin{align*}
\sqrt{W} \on{div}_g\!\left(\frac{\nabla_g F}{\sqrt W}\right)
&= \Delta_x T + \frac{2mT_r+4rT_{rr} + \frac{m}{2} \la \nabla_x f, \nabla_x T \rg + \frac{2(m-2)}{f} r f_r T_r }{f}-\frac{f \la \nabla_x T ,\nabla_x W \rg + 4r T_r W_r }{2 fW}, \\
\la \nabla_x T, \nabla_x W \rg &= 2 \, Q(T) + \frac{8 r T_r}{f} \la \nabla_x T, \nabla_x T_r \rg - \frac{4 r T_r^2}{f^2} \la \nabla_x f, \nabla_x T \rg, \\
W_r &= 2 \la \nabla_x T, \nabla_x T_r \rg + 4 \frac{T_r^2 + 2 r T_r T_{rr}}{f} - 4 r \frac{f_r}{f^2} T_r^2.
    \end{align*}
    Combining these computations, we arrive at
    \begin{align*}
        \frac{\la \nabla_x T,\nabla_x W \rg}{2}+\frac{2r T_r}{f}W_r &= Q(T) +\frac{8rT^2_r}{f^2} \left[ \frac{\la \nabla_x T , \nabla_x T_r \rg}{f^{-1} T_r} + T_r+2rT_{rr} - \frac{\la \nabla_x f, \nabla_x T \rg}{4} - \frac{rT_r f_r}{f} \right].
    \end{align*}
Finally, using $W = \frac{4 r T_r^2 + f |\nabla_x T|^2}{f}$ in the level set equation~\eqref{eqn:level-set-equation-here} proves the identity~\eqref{eqn:arrival-time-O(m)-invariant}.
\end{proof}
Due to the $O(m)$-invariance of our setting, the hypersurfaces $\cG_t = \{ (x,\xi): |\xi|^2= w(x,t) \}$ become singular when $w(x,t) = 0$, corresponding to $r=0$.
Therefore, constructing an evolving family $\{ \cG_t \}_{t < 0}$ with prescribed first-time singular set $\on{sing} \cG_t = K \times \{ 0 \} \subset \bR^n \times \bR^m$ at $t=0$ amounts to producing a modified arrival time function $T(x,r)$ in the metric $g_f$ with the property that
\begin{equation}\label{eqn:arrival-time-property}
    T(x,0) = 0 \quad \text{if } \; x \in K, \qquad T(x, s(x) ) = 0 \quad \text{if } \; x \in \bR^n \setminus K.
\end{equation}
The strictly positive function $s(x)> 0$ has the significance that in the cylindrical representation $\cG_t = \{ |\xi|^2 = w(x,t) \}$, the evolving functions $w(x,t)$ satisfy $w(x,0) = 0$ for $x \in K$ and $w(x,0) = s(x) > 0$ for $x \in \bR^n \setminus K$.
Our first observation is that the function $w_{\tau}$ of Theorem~\ref{thm:solution-of-smcf-with-estimates} produces an arrival time $T_w$ with $\cP(T_w) = 0$, and $T_w$ decays exponentially to the cylinder time $T_0(x,r) = - \frac{r}{2(m-1)}$.

\begin{proposition}\label{prop:arrival-time-construction}
Let $U \subset \bR^n$ and $h(x)$ be as in Definition~\ref{def:k-cutoff-h}.
There exists a positive function $s: U \to \bR_+$ with $s(x) < \tfrac{1}{100} h(x)^4$ and an arrival time function $T_w(x,r)$ defined in the region
\[
\{ (x,r) \in \bR^n \times [ 0,\infty) \; : \; x \in U, \; s(x) \leq r \leq 4 \, h(x)^4 \}
\]
with image $T_w(x,r) \in ( - 10 h(x)^4 , 0)$ and $T_w(x,s(x)) = 0$.
Moreover, $T_w$ satisfies
\begin{equation}\label{eqn:Tu-estimates-close}
\bigl| D^k_x \partial^{\ell}_r \bigl( T_w - T_0)(x,r) \bigr| \leq C \tau \exp( - \tfrac{4}{h(x)})    \qquad \text{in } \; \Omega_{1,3} = \{ (x,r) : x \in U, h(x)^4 \leq r \leq 3 h(x)^4 \}
\end{equation}
for each $k , \ell \in \bN_0$ and a constant $C = C(m,n,k,\ell)$.
\end{proposition}
\begin{proof}
The solution $w_{\tau}$ obtained in Theorem~\ref{thm:solution-of-smcf-with-estimates} satisfies $\partial_t w_{\tau} < -1$ and $s(x) := w_{\tau}(x,0)$ has $\exp( -\frac{1}{h(x)^2}) < s(x) < \exp( -\frac{1}{h(x)^{\frac{11}{10}}})$.
Thus, the map $t \mapsto r := w_{\tau}(x,t)$ is invertible for any fixed $x \in U$ with inverse $T_w(x,r)$, producing a well-defined, smooth arrival time function for the hypersurfaces $\cG_t = \{ (x,\xi) : T_w(x,|\xi|^2) = t \}$ with $T_w(x,w_{\tau}(x,t)) = t$ and $\cP(T_w) = 0$ on the set of reach,
    \[
    \{ (x,r) : \; x \in U, \; s(x) := w_{\tau}(x,0) \leq r \leq w_{\tau}(x, - \tfrac{1}{200m} h(x)^4) \}.
    \]
    By construction, $T_w(x,s(x)) = T_x(x,w_{\tau}(x,0)) = 0$, $\partial_r T_w  = \frac{1}{\partial_t w_{\tau}(x,T_w)} < 0$, and $\nabla_x T_w = - \frac{ \nabla_x w_{\tau}(x,T_w)}{\partial_t w_{\tau}(x,T_w)}$.

    As discussed in Remark~\ref{rmk:expand-the-domain}, the construction of $w_{\tau}$ and the subsequent interior estimates of Corollary~\ref{corollary:higher-derivative} can be implemented more generally in a domain $\{ (x,t) : x \in U , - M_0 h(x)^4 < t < 0\}$ and in sub-domains $\{ - M_2 h(x)^4 < t < - M_1 h(x)^4 \}$ for $0<M_1<M_2<M_0$, with resulting constants $C = C(m,n,M_i,k,\ell)$.
    For $|h|_{L^{\infty}(U)} \leq \tau_0$, the bound $0 < w_{\tau} + 2(m-1) t < \exp ( - h(x)^{- \frac{11}{10}})$ implies
    \begin{equation}\label{eqn:w-tau-bounds}
        w_{\tau}(x , - 8 h(x)^4) > 4 h(x)^4, \qquad w_{\tau} (x, - \tfrac{1}{100m} h(x)^4) < \tfrac{1}{2} h(x)^4,
    \end{equation}
    so repeating the above discussion with $M_0 = 10, M_2 = 9, M_1 = \frac{1}{100m}$ produces a function $T_w$ defined on the set of reach $\{ (x,r): s(x) \leq r \leq w_{\tau}(x, - 8 h(x)^4) \}$, which now contains $\{ s(x) \leq r \leq 4h(x)^4 \}$.
    Moreover, $w_{\tau}$ satisfies the estimates~\eqref{eqn:interior-estimate} in $- 5 h(x)^4 < t < - \frac{1}{100m} h(x)^4$.
    The computation~\eqref{eqn:w-tau-bounds} shows that the image of $\{ - 5 h(x)^4 < t < - \frac{1}{100m} h(x)^4 \}$ under $(x,t) \mapsto w_{\tau}(x,t)$ contains $\{ \tfrac{1}{2} h(x)^4 \leq r \leq \tfrac{7}{2} h(x)^4 \}$, and $w_{\tau}(x,t) > - 2(m-1) t = \varphi_0(t)$ implies that $T_w(x,r) > T_0(r)$.
    To prove~\eqref{eqn:Tu-estimates-close}, let $E(x,t) := w_{\tau}(x,t) - \varphi_0(t)$, so for any $(x,r) \in \Omega_{1,3}$ we have
    \begin{equation}\label{eqn:key-relation-E}
    T_w(x,r) - T_0(r) = \tfrac{1}{2(m-1)} E(x, T_w(x,r)), \qquad \sup |D^k_x \partial^{\ell}_t E| \leq C \tau \exp( - \tfrac{8}{h(x)})
    \end{equation}
    coming from $r = w_{\tau}(x, T_w(x,r)) = - 2(m-1) T_0(r)$ and~\eqref{eqn:interior-estimates-general}.
    Here, $(x,r) \in \Omega_{1,3}$ implies that $T_w(x,r) \in [ - 5 h(x)^4 , - \frac{1}{100m} h(x)^4 ]$, making the estimate~\eqref{eqn:interior-estimates-general} valid, so $|T_w - T_0| \leq C \tau \exp ( - \frac{8}{h(x)})$.

    To prove the claim~\eqref{eqn:Tu-estimates-close}, we formulate the strong induction hypothesis that
\begin{equation}\label{eqn:induction-hypothesis-new}
|D_x^k\partial_r^{\ell}(T_w-T_0)| \leq C_N \tau\, h(x)^{-M_N}\exp \bigl(-\tfrac{8}{h(x)}\bigr) \qquad \text{in } \; \Omega_{1,3}, \quad k+\ell \leq N,
\end{equation}
for $C_N, M_N$ depending only on $(m,n,N)$ and not $\tau$.
The case $N=0$ is treated above; for $N=1$, differentiating~\eqref{eqn:key-relation-E} and rearranging terms gives
    \[
    \bigl( 2(m-1) - E_t(x,T_w)) \partial_r (T_w - T_0) = - \frac{E_t(x,T_w)}{2(m-1)} , \qquad \bigl( 2(m-1) - E_t(x,T_w) \bigr) \nabla_x T_w = \nabla_xE(x,T_w),
    \]
    since $\nabla_x T_0 = 0$.
    Using the estimate~\eqref{eqn:interior-estimates-general}, we have $2(m-1) - E_t(x,T_w) > (m-1)$ for $|h|_{L^{\infty}(U)} \leq \tau_0$, hence~\eqref{eqn:induction-hypothesis-new} is proved.
Next, for $k + \ell = N \geq 2$, we apply $D^k_x \partial^{\ell}_r$ to~\eqref{eqn:key-relation-E} and obtain the top-order term
\begin{equation}\label{eqn:Rkell-inductive-step}
\bigl( 2(m-1) - E_t(x,T_w) \bigr) D^k_x \partial^{\ell}_r (T_w - T_0) = \cR_{k \ell}
\end{equation}
where every term in $\cR_{k \ell}$ is a linear combination of products involving either derivatives $D^a_x \partial^b_t E(x,T_w)$ with $a+b \geq 1$ or derivatives $D^i_x \partial^j_r (T_w - T_0)$ with $i+j \leq N-1$; no higher derivatives appear due to $D^k_x \partial^{\ell}_r T_0 = 0$ unless $(k,\ell) = (0,1)$.
The interior estimates~\eqref{eqn:interior-estimates-general} together with the strong induction hypothesis~\eqref{eqn:induction-hypothesis-new} imply that $|\cR_{k\ell}(x,r)| \leq C_{k\ell} \tau h(x)^{-M_{k\ell}} \exp ( - \frac{8}{h(x)})$, hence~\eqref{eqn:Rkell-inductive-step} proves the inductive step~\eqref{eqn:induction-hypothesis-new}.
Finally, using $h(x)^{-M_N} \exp ( - \frac{8}{h(x)}) \leq C_N \exp ( - \frac{4}{h(x)})$ completes the proof of the result.
\end{proof}

\subsection{Proof of the Main Theorem}

We now use Proposition~\ref{prop:arrival-time-construction} to construct the evolving hypersurfaces $\{ \cG_t \}$ considered in Theorem~\ref{thm:main-theorem} via their arrival time functions.
To this end, we glue the arrival time function $T_w$ produced above to the cylinder time $T_0$, obtaining an arrival time $T$ on the swept-out region that satisfies the property~\eqref{eqn:arrival-time-property}.
This will produce an approximate solution of equation $\cP(T) = 0$, which we then realize as an exact solution of $\cP_f(T) = 0$ in the modified metric $g_f$, for a function $f(x,r)$ satisfying estimates of the form $|D^k_x \partial^{\ell}_r(f-1)| < C_{k,\ell} \tau$ for all $k,\ell \geq 0$.
\begin{proof}[Proof of Theorem~\ref{thm:main-theorem}]
We fix a $C^{\infty}$ cutoff function with the properties $\eta(t) = 1$ for $t \leq 2$, $\eta(t) = 0$ for $t \geq 3$, $|\eta'(t)| + |\eta''(t)| \leq 20$ and $|D^k_t \eta(t)| \leq C_k$ for $k \geq 3$.
For $\tau \in (0,\tau_0]$ sufficiently small so that Proposition~\ref{prop:arrival-time-construction} applies, we take $T_w$ and $s(x)$ as constructed therein and define the function
\begin{equation}\label{eqn:arrival-time-function}
    T(x,r) := \begin{cases}
        T_0(r), & (x,r) \in K \times [0, \infty), \\
        \eta \bigl( \frac{r}{h(x)^4} \bigr) T_w(x,r) + \bigl( 1 - \eta \bigl( \frac{r}{h(x)^4} \bigr) \bigr) T_0(r), & x \in U, \; r \geq s(x).
    \end{cases}
\end{equation}
By construction, for $x \in U$ and $s(x) \leq r \leq 4 h(x)^4$ we can use the expression~\eqref{eqn:key-relation-E} to write
\begin{equation}\label{eqn:enjoy-the-estimates-for-T}
T(x,r) - T_0(r) = \eta \bigl( \tfrac{r}{h(x)^4} \bigr) \bigl( T_w(x,r) - T_0(r) \bigr) = \tfrac{1}{2(m-1)} \eta \bigl( \tfrac{r}{h(x)^4} \bigr) E(x, T_w(x,r)),
\end{equation}
so that $T(x,r)$ is well-defined for $\{ x \in U, r \geq s(x) \}$ even beyond the region where $T_w$ exists, namely $r \geq 4 h(x)^4$, due to $\eta( \frac{r}{h(x)^4}) = 0$ and $T(x,r) = T_0(r)$ for $r \in [ 4h(x)^4, \infty)$.
Moreover, $T(x,r) = T_w(x,r)$ for $s(x) \leq r \leq 2 h(x)^4$.
To summarize, $T(x,r)$ satisfies
\[
\cP(T) = 0 \qquad \text{for } \; (x,r) \in \bigl( K \times [0,\infty) \bigr) \; \cup \; \bigr\{ (x,r): x \in U, r \in [s(x) , 2 h(x)^4] \cup [ 3 h(x)^4, \infty) \bigr\}.
\]
Here, $\cP = \cP_1$ denotes the Euclidean arrival time operator, which for $f=1$ assumes the form
\begin{align*}
      \cP(T) &:= 1 + \Delta_x T + 2 (T_r + 2r T_{rr}) - \frac{Q(T) + 8 r T_r \la \nabla_x T, \nabla_x T_r \rg + 8 r T_r^2(T_r + 2 r T_{rr})}{4 r T_r^2 + |\nabla_x T|^2} + 2(m-1) T_r.
\end{align*}
Moreover, $T(x,r)$ enjoys the closeness estimates~\eqref{eqn:Tu-estimates-close} to the cylinder time $T_0(r)$ in the region $\{ h(x)^4 \leq r \leq 3 h(x)^4 \}$.
For the function $f$ that appears in the metric $g_f$ of~\eqref{eqn:general-metric}, we let $z(x,r) := \frac{1 - f(x,r)}{r}$ with $(\nabla_x z, \partial_r z) = - (\nabla_x f, \frac{\partial_r f}{r} - \frac{f}{r^2} )$, so the arrival time equation $\cP_f(T) = 0$ from~\eqref{eqn:arrival-time-O(m)-invariant} becomes
\begin{equation}\label{eqn:final-transport-equation}
\cP(T)-rz\bigl(\cE(T,z)+\cP(T)\bigr)-r\cA(T,z)\langle \nabla_x z,\nabla_x T\rangle-\cB(T,z) (z + rz_r) = 0.
\end{equation}
Letting $\cR(T,z) := \frac{4 r T_r^2}{4 r T_r^2 + (1 - r z ) |\nabla_x T|^2}$, we obtain the coefficients of equation~\eqref{eqn:final-transport-equation} as
\allowdisplaybreaks{
\begin{align*}
    \cA(T,z) &:= \frac{1}{2} \bigl(m+\cR(T,z) \bigr), \qquad \cB(T,z) := \frac{2rT_r}{1-rz} \bigl(m-2+\cR(T,z) \bigr), \\
    \cE(T,z) &= -2\bigl(T_r+2rT_{rr}\bigr)-2(m-1)T_r \\
    &\quad -\frac{(1-rz)\bigl(Q(T)T_r-2|\nabla_xT|^2\langle \nabla_xT,\nabla_xT_r\rangle\bigr) - 2 T_r \bigl(T_r+2rT_{rr}\bigr)\bigl(4rT_r^2+(2-rz)|\nabla_xT|^2\bigr) }{ T_r \bigl(4rT_r^2+|\nabla_xT|^2\bigr)} \cR,
\end{align*}}
where $Q(T) = T_{x_i} T_{x_j} T_{x_i x_j}$.
For the cylinder time $T_0(r) = - \frac{r}{2(m-1)}$, we have $(T_0)_r = - \frac{1}{2(m-1)}$ and $D^k_x \partial^{\ell}_r T_0 = 0$ unless $(k,\ell) = (0,1)$.
Thus, $\cR(T_0,z) = 1, \cP(T_0) = 0$, and equation~\eqref{eqn:final-transport-equation} becomes
\begin{equation}\label{eqn:cylinder-equation-terms}
    \cA(T_0, z) = \tfrac{m+1}{2}, \quad \cB(T_0,z) = - \tfrac{r}{1-rz}, \quad \cE(T_0,z) = 1, \qquad z_r + z^2 = 0. 
\end{equation}
Using these explicit terms, we rewrite equation~\eqref{eqn:final-transport-equation} in the form
\begin{equation}\label{eqn:simplified-in-gluing-region}
    z_r + z^2 = - \la a(x,r,z), \nabla_x z \rg + b(x,r,z) \, z + c(x,r,z),
\end{equation}
where we introduced the terms
\begin{align*}
    a(x,r,z) &:= \tfrac{\cA(T,z)}{\cB(T,z)} \nabla_x T, \qquad b(x,r,z) := - \tfrac{\cB(T,z) (1-rz) + r \cE(T,z) + r \cP(T)}{r \cB(T,z)}, \qquad c(x,r,z) := \tfrac{\cP(T)}{r \cB(T,z)}.
\end{align*}
In view of the estimate~\eqref{eqn:Tu-estimates-close}, which also applies to the difference~\eqref{eqn:enjoy-the-estimates-for-T} in the gluing region $\Omega_{1,3} = \{ h(x)^4 \leq r \leq 3 h(x)^4 \}$, we have $1-rz \in [ \frac{1}{2}, \frac{3}{2}]$ for $|h|_{L^{\infty}(U)} \leq \tau_0$ and $|z| < \frac{1}{2}$.
Letting $\Omega_{1,3} := \{ x \in U, 1 \leq \frac{r}{h(x)^4} \leq 3 \}$, we therefore find that the coefficients $a,b,c$ satisfy the pointwise bounds
\begin{equation}\label{eqn:bounds-on-a,b,c}
    r^k |D^k_{x,r,z} a| + r^{k+1} |D^k_{x,r,z} b| + r^{k+1} |D^k_{x,r,z} c| \leq C_k \, \tau \, \exp \bigl( - \tfrac{2}{h(x)} \bigr) \quad \text{for } \; (x,r) \in \Omega_{1,3}, \; |z| < \tfrac{1}{2}.
\end{equation}
The bound~\eqref{eqn:bounds-on-a,b,c} follows from applying the estimates~\eqref{eqn:Tu-estimates-close} from Proposition~\ref{prop:arrival-time-construction}, followed by the bound $h(x)^{-M} \exp( - \frac{4}{h(x)}) \leq C_M \exp( - \frac{2}{h(x)})$ to iteratively absorb derivative terms in $(x,r,z)$ as done in Proposition~\ref{prop:arrival-time-construction}.
In particular, the parameter $\tau \in (0,\tau_0]$ can be constructed with $\tau_0 = \tau_0(m,n)$ sufficiently small, so that $|a| < \frac{1}{10}$ in~\eqref{eqn:bounds-on-a,b,c} while $|z| < \frac{1}{2}$, hence the transport equation defines a non-degenerate quasilinear first order PDE for $|z|<\frac{1}{2}$.
With the estimates~\eqref{eqn:bounds-on-a,b,c} in place, we can apply Proposition~\ref{prop:solvability-of-transport} to obtain a smooth solution $z$ of equation~\eqref{eqn:simplified-in-gluing-region} in the region $\{ h(x)^4 \leq r \leq 3 h(x)^4 \}$ with $\{ z= 0 \}$ on the initial hypersurface $\Sigma = \{ (x,h(x)^4) : x \in U \}$, which satisfies the bounds
\begin{equation}\label{eqn:solution-z-bounds}
    |D^k_{x,r} z(x,r)| \leq C(m,n,k) \tau \exp( - \tfrac{3}{2} \tfrac{1}{h(x)} )\qquad \text{in } \; \Omega_{1,3} = \{ (x,r) : x \in U, \; h(x)^4 \leq r \leq 3 h(x)^4 \}.
\end{equation}
Since $T = T_w$ with $\cP(T) = 0$ for $h(x)^4 \leq r \leq 2 h(x)^4$, the uniqueness theorem for solutions of the initial value problem for this first-order quasilinear PDE forces $z \equiv 0$ on $\{ h(x)^4 \leq r \leq 2 h(x)^4 \}$.
In the region $r \geq 3 h(x)^4$, we have $T(x,r) = T_0(r)$, so equation~\eqref{eqn:final-transport-equation} becomes the cylinder ODE~\eqref{eqn:cylinder-equation-terms}, which is separable with explicit solution $z(x,r) = \frac{G(x)}{1 + G(x) r}$, hence $f(x,r) = 1- r z(x,r)= \frac{1}{1+G(x) r}$ for a function $G(x)$.
The hypersurface $\{ r = 3 h(x)^4 \}$ provides the initial condition
\begin{align*}
G(x) &= \frac{z(x , 3 h(x)^4)}{1 - 3 h(x)^4 z( x, 3h(x)^4)}, \qquad f(x,r) = \frac{1 - 3 h(x)^4 z(x,3h(x)^4)}{1 - 3 h(x)^4 z(x,3h(x)^4) + z(x,3h(x)^4) r}.
\end{align*}
Here, $|D^k_x G(x)| \leq C(m,n,k) \tau \exp \bigl( - \tfrac{1}{h(x)} \bigr)$ due to~\eqref{eqn:solution-z-bounds}, thus $G$ extends smoothly by $0$ across $K$.
We conclude that $f(x,r) = 1 - r z(x,r)$ is defined and smooth on the entire region $\{ (x,r) : x \in U, r \geq s(x) \}$, satisfying the desired properties with $\sup |D^k_x \partial^{\ell}_r (f-1)| < C(m,n,k, \ell) \tau$ for all $k,\ell$ and
\[
f \equiv 1 \quad \text{on} \quad (K \times \bR_{\geq 0}) \cup \{ (x,r) : x \in U, \; 0 \leq r \leq 2 h(x)^4 \}.
\]
Since $\cP_f(T) = 0$, the hypersurfaces $\cG_t = \{ (x,\xi) : T(x,|\xi|^2) = t \}$ evolve by mean curvature under the metric $g_f$, are mean-convex due to the presence of a well-defined arrival time, and are defined for all $t<0$.
The definition~\eqref{eqn:arrival-time-function} ensures that $T(x,0) = 0$ for $x \in K$ and $T(x,s(x)) = 0$ with $s(x) > 0$ for $x \in U$, hence $\{ \cG_t \}_{t<0}$ is a mean-convex ancient mean curvature flow that becomes singular at time $t=0$, when it encounters precisely the set $\on{sing} \cG_t = K \times \{ 0 \} \subset \bR^{m+n}$.
\end{proof}

\appendix

\section{Solvability of a transport equation}\label{appendix:solvability-transport}

We solve here the transport equation~\eqref{eqn:simplified-in-gluing-region} to construct the warping function $f(x,r)$ of Theorem~\ref{thm:main-theorem}.

\begin{proposition}\label{prop:solvability-of-transport}
Let $U$ and $h(x)$ be as in Section~\ref{section:non-cylindrical} and define the domain $Q := \{ (x,r,z) : x \in U, h(x)^4 \leq r \leq 3h(x)^4, |z| \leq \frac{1}{2} \}$.
We consider the transport equation
\begin{equation}\label{eqn:transport-equation}
    z_r + z^2 = - \la a(x,r,z), \nabla_x z \rg + b(x,r,z) \, z + c(x,r,z),
\end{equation}
where $a,b,c$ satisfy the bounds~\eqref{eqn:bounds-on-a,b,c}.
For $\tau_0(m,n)$ small and $\tau \in (0,\tau_0]$, the initial value problem for equation~\eqref{eqn:transport-equation} with data $\{ z=0 \}$ on $\Sigma = \{ (x,h(x)^4) : x \in U \}$ has a unique solution, which satisfies
\begin{equation}\label{eqn:z-bounds-now}
    |D^k_{x,r} z(x,r)| \leq C(m,n,k) \tau \exp \bigl( - \tfrac{3}{2} \tfrac{1}{h(x)} \bigr) \qquad \text{in } \; \{ x \in U, \; h(x)^4 \leq r \leq 3 h(x)^4 \}.
\end{equation}
\end{proposition}
\begin{proof}
We use a local rescaling and the method of characteristics as in~\cite{prescribed-singular}*{\S 3}; our choice of $z = \frac{1-f}{r}$ simplifies the transport equation~\eqref{eqn:transport-equation} and its solution significantly.
The key property to verify will be the uniform estimate~\eqref{eqn:z-bounds-now} near $\partial U$ used in the gluing construction of Theorem~\ref{thm:main-theorem}.

We fix an $x_0 \in U$, set $\rho := h(x_0)^4$, and define $\Psi(y) := \frac{h(x_0 + \rho y)^4}{\rho}$ for $(x,r) = (x_0 + \rho y, \rho s)$.
For small $\tau_0$, $|y| \leq 6$, and all $k \geq 1$, Lemma~\ref{lemma:geometric-properties-of-h} gives
\begin{equation}\label{eqn:local-h-comparison}
\tfrac{1}{2} h(x_0) \leq h(x_0+\rho y) \leq 2h(x_0), \qquad \tfrac{1}{2} \leq \Psi(y)\leq 2, \qquad |D_y^k \Psi(y)| \leq C_k\tau.
\end{equation}
Let $w(y,s) := z(x_0 + \rho y, \rho s)$, which has initial condition $w(y,\Psi(y)) = 0$ and $(\nabla_x z, z_r) = \rho^{-1} (\nabla_y w, w_s)$.
The $\rho$-rescaled equation~\eqref{eqn:transport-equation} is written in terms of $w$ as
\begin{equation}\label{eqn:rescaled-equation-for-w}
    w_s + \rho w^2 = - \la \tilde{a}(y,s,w), \nabla_y w \rg + \tilde{b}(y,s,w) \, w + \tilde{c}(y,s,w), \qquad w(y,\Psi(y)) = 0.
\end{equation}
Here, we defined the coefficients $\tilde{a}, \tilde{b}, \tilde{c}$ as rescalings of $a,b,c$ via
\[
\tilde{a}(y,s,w) := a(x_0 + \rho y, \rho s,w), \quad \tilde{b}(y,s,w) := \rho b(x_0 + \rho y, \rho s, w), \quad \tilde{c}(y,s,w) := \rho c(x_0 + \rho y, \rho s, w).
\]
Combining the bounds~\eqref{eqn:bounds-on-a,b,c} with $h(x_0 + \rho y) < \frac{20}{19} h(x_0)$ for $|y| \leq 6$ due to Lemma~\ref{lemma:geometric-properties-of-h}, we obtain
\begin{equation}\label{eqn:rescaled-coefficients}
    |D^k_{y,s,w}\tilde a|+|D^k_{y,s,w}\tilde b|+|D^k_{y,s,w}\tilde c| \leq C_k \tau e^{- \frac{1.9}{h(x_0)}} \quad \text{in } \; \tilde{Q} := \bigl\{ (y,s,w) : |y| \leq 6, s \in [ \tfrac{1}{2}, 6], \, |w| \leq \tfrac{1}{2} \bigr\}.
\end{equation}
For $\eta \in B_5 := B^n_5(0)$ and $t \in [0,4]$, we set $S(\eta,t) := t + \Psi(\eta)$ and consider the characteristic system
\begin{equation}\label{eqn:char-system-final}
\begin{cases}
\partial_t Y(\eta,t)=\tilde a(Y,S,W),\\
\partial_t W(\eta,t)=\tilde b(Y,S,W)\,W+\tilde c(Y,S,W)-\rho\,W^2,
\end{cases} \qquad Y(\eta,0)=\eta,\quad W(\eta,0)=0.
\end{equation}
We claim that for $\tau_0$ small, the system~\eqref{eqn:char-system-final} has a smooth solution $(Y,W)$ on all of $B_5 \times [0,4]$ that remains inside $\tilde{Q}$.
We will then prove that the characteristic map $\Xi: B_5 \times [0,4] \to \bR^n \times \bR$ of the system, given by $\Xi(\eta,t) = (Y(\eta,t), S(\eta,t))$, maps $B_5 \times [0,4]$ diffeomorphically onto a set containing $\cS := \{ |y|<4, \Psi(y) \leq s \leq 3 \Psi(y) \}$.
After proving that $\Xi$ is smoothly invertible, we see by direct differentiation that given any solution $(Y,W)$ of~\eqref{eqn:char-system-final}, the function $w(y,s) = W(\Xi^{-1}(y,s))$ defines a smooth solution of~\eqref{eqn:rescaled-equation-for-w}.
Since the local smooth existence of a solution to~\eqref{eqn:char-system-final} is standard as long as $(Y,S,W) \in \tilde{Q}$, by a continuation argument it suffices to prove that $|W| \leq \frac{1}{4}$ for all $(\eta,t) \in B_5 \times [0,4]$.

For $|h|_{L^{\infty}(U)} \leq \tau_0$ small, the $Y$-equation~\eqref{eqn:char-system-final} and the bound~\eqref{eqn:rescaled-coefficients} for $\tilde{a}$ give
\begin{equation}\label{eqn:y-bound}
    |\partial_t Y| \leq C \tau \exp \bigl( - \tfrac{1.9}{h(x_0)} \bigr), \quad \implies \quad |Y(\eta,t) - \eta| \leq C \tau \exp \bigl( - \tfrac{1.9}{h(x_0)} \bigr), \quad \implies \quad |Y(\eta,t)| \leq \tfrac{11}{2},
\end{equation}
while $S(\eta,t) = t + \Psi(\eta)$ has $\frac{1}{2} \leq S(\eta,t) \leq 6$ by~\eqref{eqn:local-h-comparison}.
Consequently,~\eqref{eqn:rescaled-coefficients} and ~\eqref{eqn:char-system-final} imply
\[
|\partial_t W| \leq C \tau \exp \bigl( - \tfrac{1.9}{h(x_0)} \bigr) ( 1 + |W|) + \rho |W|^2, \qquad \rho = h(x_0)^4 \leq \tau_0^4.
\]
Starting from $W(\eta,0) = 0$, a standard bootstrap together with Gr\"onwall's inequality yields
\begin{equation}\label{eqn:W(t,eta)-bound}
    |W(\eta,t)| \leq C \tau \exp ( - \tfrac{1.8}{h(x_0)}), \qquad \text{for } \; (\eta,t) \in B_5 \times [0,4].
\end{equation}
Therefore, $|W| \leq \frac{1}{4}$ for $\tau_0$ small, so the solution of~\eqref{eqn:char-system-final} exists on all of $B_5 \times [0,4]$ and remains in $\tilde{Q}$.

Using the properties~\eqref{eqn:y-bound} and~\eqref{eqn:W(t,eta)-bound} of $(Y,S,W)$, we can express the characteristic map $\Xi$ as
\[
\Xi(\eta,t) = (\eta, t + \Psi(\eta)) + E(\eta,t) = \Xi_0(\eta,t) + E(\eta,t), \qquad \Xi_0(\eta,t) := (\eta, t+\Psi(\eta)), 
\]
where $|E(\eta,t)| \leq C \tau e^{- \frac{1.8}{h(x_0)}}$.
Differentiating the system \eqref{eqn:char-system-final} repeatedly with respect to $(\eta,t)$ and using~\eqref{eqn:local-h-comparison},~\eqref{eqn:rescaled-coefficients}, and \eqref{eqn:W(t,eta)-bound}, the standard smooth dependence theorem for ODE gives
\begin{equation}\label{eqn:derivative-bounds-char}
|D_{\eta}^j D_t^k Y(\eta,t)| \leq C_{jk}, \quad |D_{\eta}^j D_t^k W(\eta,t)| \leq C_{jk} \tau \exp( - \tfrac{1.8}{h(x_0)} ), \qquad \text{for all } \; (\eta,t) \in B_5 \times [0,4].
\end{equation}
Combined with~\eqref{eqn:y-bound}, this gives $\| D \Xi - D \Xi_0 \|_{L^{\infty}(B_5 \times [0,4])} \leq C \tau e^{- \frac{1.8}{h(x_0)}} < \frac{1}{10}$ for $\tau_0$ small.
Since $\Xi_0$ is a smooth diffeomorphism $B_5 \times [0,4] \xrightarrow{\cong} \{ (y,s) : |y| < 5, \Psi(y) \leq s \leq \Psi(y) + 4 \}$, we deduce that $\Xi$ maps $B_5 \times [0,4]$ diffeomorphically onto a set containing $\cS := \{ (y,s) : |y| < 4, \Psi(y) \leq s \leq 3 \Psi(y) \}$.
Moreover, differentiating the identity $\Xi^{-1}\circ \Xi=\mathrm{Id}$ repeatedly and using \eqref{eqn:derivative-bounds-char}, we obtain
\begin{equation}\label{eqn:DkXiy,s-estimates}
    |D^k_{y,s} \Xi^{-1}(y,s)| \leq C_k, \qquad \text{for all } \; (y,s) \in \cS.
\end{equation}
Finally, we define $w(y,s) := W(\Xi^{-1}(y,s))$ for $(y,s) \in \cS$, which is a smooth solution of~\eqref{eqn:rescaled-equation-for-w} with $w(\eta, \Psi(\eta)) = 0$.
Recalling the estimates~\eqref{eqn:derivative-bounds-char} and~\eqref{eqn:DkXiy,s-estimates} and differentiating repeatedly, we find
\begin{equation}\label{eqn:Dkys-w-bounds-finally}
    |D^k_{y,s} w(y,s)| \leq C_k \tau \exp( - \tfrac{1.7}{h(x_0)}), \qquad \text{for all } \; (y,s) \in \cS.
\end{equation}
Returning to $z = w ( \frac{x-x_0}{\rho}, \frac{r}{\rho})$, we obtain a smooth local solution $z$ of~\eqref{eqn:transport-equation} in the neighborhood
\[
\cU_{x_0} := \{ (x,r) : |x  - x_0| < 4 h(x_0)^4, \; h(x)^4 \leq r \leq 3 h(x)^4 \}.
\]
The derivative estimates~\eqref{eqn:Dkys-w-bounds-finally} for $w$ yield $|D^k_{x,r} z(x,r)| \leq C_k \rho^{-k} \tau \exp( - \frac{1.7}{h(x_0)})$ with $\rho = h(x_0)^4$, so
\[
|D^k_{x,r} z| \leq C_k \tau \exp ( - \tfrac{1.6}{h(x_0)}) \leq C_k \tau \exp \bigl( - \tfrac{3}{2} \tfrac{1}{h(x)} \bigr) \qquad \text{in } \; \cU_{x_0},
\]
due to $h(x) > \frac{15}{16} h(x_0)$ for $x \in \cU_{x_0}$ by Lemma~\ref{lemma:geometric-properties-of-h}.
This proves the uniform bound~\eqref{eqn:z-bounds-now} for local solutions, whose domains $\cU_{x_0}$ cover $\Omega_{1,3} = \{ 1 \leq \frac{r}{h(x)^4} \leq 3 \}$ and coincide on overlaps, by uniqueness for the system~\eqref{eqn:char-system-final}.
Therefore, the local solutions we constructed assemble into a unique global solution $z$ in $\Omega_{1,3}$ with $|z| < \frac{1}{2}$, which satisfies the uniform estimates~\eqref{eqn:z-bounds-now} as desired.
\end{proof}

\bibliography{ref}

\end{document}